\documentclass[10pt,a4paper,reqno]{amsart} 

\usepackage[english]{babel}
\usepackage[utf8]{inputenc}
\usepackage{hyperref}
\usepackage{latexsym}
\usepackage{verbatim}
\usepackage{rotating}
\usepackage{color,bbm}
\usepackage{amssymb,amsfonts,amsmath,stmaryrd}
\usepackage[T1]{fontenc}

\usepackage[normalem]{ ulem }
\usepackage{soul}

\newcommand{\Bgf}{{\sf B}}
\newcommand{\Cgf}{{\sf C}}
\newcommand{\Dgf}{{\sf D}}
\newcommand{\Ggf}{{\sf G}}

\newcommand{\Pgf}{{\sf P}}
\newcommand{\Qgf}{{\sf Q}}
\newcommand{\Qc}{{\sf Q^c}}
\newcommand{\bR}{\bar R}
\newcommand{\Rgf}{{\sf R}}
\newcommand{\Ugf}{{\sf U}}

\newcommand{\Cnn}{{\mathcal C}}
\newcommand{\Dnn}{{\mathcal D}}
\newcommand{\Pnn}{{\mathcal P}}

\DeclareMathOperator{\Arg}{Arg}
\DeclareMathOperator{\Cat}{Cat}
\DeclareMathOperator{\Tpol}{T}

\newcommand{\notem}[1]{\textcolor{blue}{#1}}

\newtheorem{Theorem}{Theorem}[section]
\newtheorem{Lemma}[Theorem]{Lemma}
\newtheorem{Proposition}[Theorem]{Proposition}
\newtheorem{Definition}[Theorem]{Definition}
\newtheorem{Corollary}[Theorem]{Corollary}

\newcommand{\beq}{\begin{equation}}
\newcommand{\eeq}{\end{equation}}
\def\emm#1,{{\em #1}}

\catcode`\@=11
\def\section{\@startsection{section}{1}%
 \z@{.7\linespacing\@plus\linespacing}{.5\linespacing}%
 {\normalfont\bfseries\scshape\centering}}

\def\subsection{\@startsection{subsection}{2}%
  \z@{.5\linespacing\@plus\linespacing}{.5\linespacing}%
  {\normalfont\bfseries\scshape}}

\def\subsubsection{\@startsection{subsubsection}{3}%
 \z@{.5\linespacing\@plus\linespacing}{-.5em}
 {\normalfont\bfseries}}
\catcode`\@=12

\newcommand{\zs}{\mathbb{Z}}

\newcommand{\qs}{\mathbb{Q}}

\newcommand{\fps}{formal power series}

\DeclareMathOperator{\cc}{c}
\DeclareMathOperator{\vv}{v}
\DeclareMathOperator{\ff}{f}
\DeclareMathOperator{\ee}{e}

\DeclareMathOperator{\od}{od}
\DeclareMathOperator{\dig}{dig}
\DeclareMathOperator{\qu}{qu}
\newcommand{\gf}{generating function}
\newcommand{\gfs}{generating functions}

\newcommand{\om}{\omega} 
\renewcommand{\epsilon}{\varepsilon}
\newcommand{\vareps}{\varepsilon}

\graphicspath{{Figures/}}

\begin{document}
\title[The generating function  of planar Eulerian orientations]
{The  generating function   of planar Eulerian orientations}

\author[M. Bousquet-M\'elou]{Mireille Bousquet-M\'elou}

\author[A. Elvey Price]{Andrew Elvey Price}

\thanks{Both authors  were partially supported by the French ``Agence Nationale
de la Recherche'',  via grant  Graal ANR-14-CE25-0014. AEP was also
supported by ACEMS in the form of a top up scholarship and travel stipend, a 2017 Nicolas Baudin travel grant and an Australian government research training program scholarship.}

\address{MBM: CNRS, LaBRI, Universit\'e de Bordeaux, 351 cours de la
  Lib\'eration,  F-33405 Talence Cedex, France}

\address{
  AEP: {{School of mathematics and statistics}, University of Melbourne, Parkville, Victoria 3010, Australia}} 
\email{bousquet@labri.fr, andrewelveyprice@gmail.com}

\begin{abstract}The enumeration of planar maps equipped with an
  Eulerian orientation has attracted attention in both
  combinatorics and theoretical physics since at least 2000. The case
  of 4-valent maps is particularly interesting: these orientations are in
  bijection with properly 3-coloured quadrangulations, while in
  physics they correspond  to configurations   of the ice model.

We solve both problems -- namely the enumeration of  planar
Eulerian orientations and of 4-valent planar Eulerian orientations --
by expressing the associated generating functions as the
 inverses (for the composition of series) of simple hypergeometric
 series. Using these expressions, we derive the asymptotic behaviour of
 the number of planar Eulerian orientations, thus proving earlier
 predictions of Kostov, Zinn-Justin, Elvey Price and Guttmann. This
 behaviour,  $\mu^n /(n \log n)^2$, prevents the associated generating functions
 from being D-finite. Still, these generating functions are differentially algebraic, as they satisfy non-linear differential
 equations of order $2$. Differential algebraicity has recently been
 proved  for other map problems, in particular for maps equipped with
 a Potts model.

Our solutions mix recursive and bijective ingredients. In particular, a preliminary  bijection transforms our oriented maps into  maps carrying  a height function on their vertices.
In the 4-valent case, we also observe an unexpected connection with the
enumeration of maps equipped with a spanning tree that is \emm internally
inactive, in the sense of Tutte. This connection remains to be
explained combinatorially.
\end{abstract}

\keywords{planar maps, Eulerian orientations, height functions, differentially algebraic series}
\maketitle

\section{Introduction}
A \emm planar map, is a connected planar graph
embedded in the sphere, and taken up to orientation preserving
homeomorphism (see Figure~\ref{fig:defs}).  
The enumeration of planar maps is a venerable topic in combinatorics, which was born in
the early sixties with the pionneering work of William Tutte~\cite{tutte-triangulations,tutte-census-maps}. Fifteen
years later it started a second, independent, life in theoretical
physics, where planar maps are seen as a discrete model of \emm quantum
gravity,~\cite{BIPZ,BIZ}. The enumeration of maps also has connections with
factorizations of permutations, and hence representations of the
symmetric group~\cite{Jackson:Harer-Zagier,jackson-visentin}. Finally, 40 years after the first
enumerative results of Tutte, planar maps  crossed the border between
combinatorics and probability theory, where they are now studied as
 random metric spaces~\cite{angel-schramm,chassaing-schaeffer,le-gall-topological,marckert-mokkadem}. The limit behaviour of large
planar random maps is now well understood, and gave birth to a
variety of 
 limiting objects, either
continuous like the Brownian map~\cite{curien-legall-plane,legall,legall-miermont-large-faces,miermont}, or
 discrete  like the UIPQ (uniform infinite planar quadrangulation)~\cite{angel-schramm,chassaing-durhuus,curien-miermont,menard-same}.

The enumeration of maps equipped with some additional structure (a
spanning tree, a proper colouring, a self-avoiding-walk, a configuration of the Ising
model...) {has} attracted the {interest} of both combinatorialists and theoretical
physicists {since} the early days of this
study~\cite{DK88,Ka86,mullin-boisees,lambda12,tutte-dichromatic-sums}. {At
  the moment}, {a challenge}
is to understand the {limiting} behaviour of maps equipped with one such
structure~\cite{borot-bouttier-duplantier,kassel-wilson,kenyon2015bipolar,richier-perco,sheffield-inventory}.

\begin{figure}[ht]
  \centering
  \includegraphics[width=12cm]{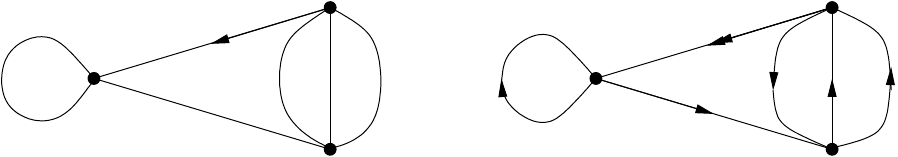}
  \caption{Left: a rooted planar map, which is 4-valent (or:
    quartic). Right: the same map, equipped with an Eulerian
    orientation.}
  \label{fig:defs}
\end{figure}

The enumeration of {these} ``decorated'' maps, and  understanding  their
structure, remain the very first building blocks towards the
resolution of such
challenges. Recently, the natural question of counting maps equipped
with an \emm Eulerian orientation, (where all edges are oriented in such a way
that every vextex has as many incoming as outgoing edges, see Figure~\ref{fig:defs}) was raised by
Bonichon \emph{et al.}~\cite{BoBoDoPe}. They did not solve the problem, but gave
sequences of lower bounds and upper bounds on the number of 
 planar Eulerian orientations. They were followed by
Elvey Price and Guttmann who, remarkably, were able to write an
intricate system of functional equations defining the associated
\gf~\cite{elvey-guttmann17}. This allowed them to compute the number $g_n$ of Eulerian orientations
with $n$ edges for large values of $n$, and led them to a conjecture
on the asymptotic behaviour  of $g_n$.

Their study {also} included the special case of 4-valent (or:
\emm quartic,) Eulerian orientations. This problem  had already been
studied around 2000 in theoretical
physics, where it coincides with the \emm ice model on a random
lattice,~\cite{kostov,zinn-justin-6V-random}.  Another fact that makes this case
particularly relevant is that the number of such orientations with $n$
vertices is known to be the number of 3-coloured quadrangulations with
$n$ faces~\cite{welsh-book}. Elvey Price and Guttmann constructed a system
of functional equations for this problem {as well}, and conjectured the
asymptotic behaviour of the associated numbers $q_n$. Their prediction
had already been in the physics
papers~\cite{kostov,zinn-justin-6V-random} for a while, but was probably less accessible
to combinatorialists. The more experienced  of us
observed that the conjectured  growth rate, $4\sqrt 3 \pi$,  already
occurred when counting quartic maps equipped with a certain type of tree~\cite{mbm-courtiel}, and the
more optimistic of us looked for, and discovered, an exact (though conjectural)
 relation between the two problems. This gave an (unpublished) conjecture for the
 \gf\ of quartic Eulerian orientations, soon completed by a similar conjecture
 for general Eulerian orientations. These are the conjectures that we
 prove in this paper, thus {completely solving} these two enumeration
 problems. 

 Let us now state our main two theorems.   As is usual with maps, our orientations are \emm rooted,,
 which means that we mark one (oriented) edge (Figure~\ref{fig:defs},
 {right}). 
 Orientations of small size are shown in
 Figure~\ref{fig:small}.

\begin{figure}[ht]
  \centering
  \includegraphics[scale=0.7]{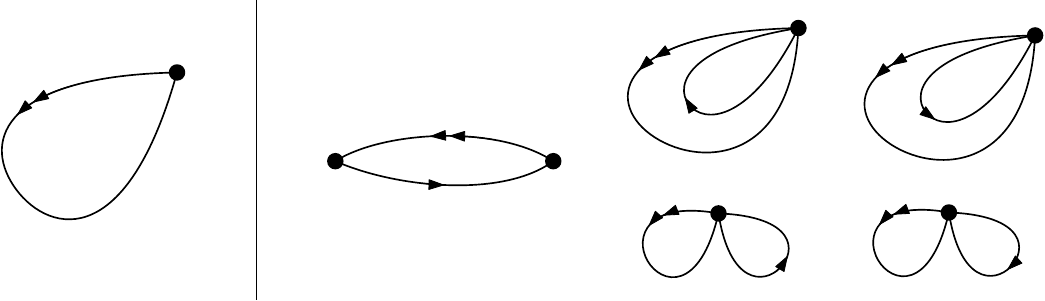}
  \caption{{The planar Eulerian orientations with at most
      two edges, in agreement with 
      $\Ggf(t)=t+5t^2+O(t^3)$. On the right are the four quartic Eulerian
      orientations with one vertex, in agreement with $\Qgf(t)=4t+O(t^2)$.}}
  \label{fig:small}
\end{figure}

\begin{Theorem}\label{thm:4}
  Let $\Rgf(t)\equiv \Rgf$ be the unique formal power series with constant
  term $0$ satisfying
\[
t= \sum_{n \ge 0} \frac{1}{n+1} {2n \choose n}{3n \choose n} \Rgf^{n+1}.
\]
Then the \gf\ of  quartic rooted planar Eulerian orientations, counted by
vertices, is
\[
\Qgf(t)= \frac 1{3t^2}\left( t-3t^2-\Rgf(t)\right).
\]
This is a differentially algebraic series, satisfying a non-linear
differential equation of order $2$ whose coefficients are polynomials in $t$. The number $q_n$ of such
orientations having $n$ vertices behaves asymptotically as 
\[
q_n \sim \kappa\, \frac{\mu^{n+2} }{n^2 (\log n)^2}
\]
where
\[ \kappa =1/18  \qquad \hbox{and } \qquad \mu=  4\sqrt 3\, \pi.
\]
The series $\Qgf(t)$ is not D-finite, which means that it does not
satisfy any non-trivial \emm linear, differential equation.
\end{Theorem}
The first coefficients of $\Rgf$ and $\Qgf$ are
$$
\Rgf(t)= t-3t^2-12t^3-105t^4 -1206 t^5- \cdots, \qquad 
\Qgf(t)=4t+35t^2+402 t^3+ \cdots.
$$

\medskip
\noindent{\bf Remarks}\\
{\bf 1.} As we will explain in Section~\ref{sec:def-orientations},
  the series $\Qgf(t)$  also counts (by faces)
quadrangulations  equipped with a {proper} 3-colouring
of the vertices (with prescribed colours on  the root edge).
It is worth noting that the \gfs\ of 3-coloured triangulations, and of
3-coloured planar maps, are both algebraic~\cite{bernardi-mbm-alg} (and thus D-finite), hence in a sense
{they are} much simpler.  The corresponding asymptotic estimates are 
$\kappa \mu ^n n^{-5/2}$ in both cases (for other values of $\mu$ {and~$\kappa$} of course).\\
{\bf 2.} In Section~\ref{sec:bij}, {we will prove} that $\Qgf(t)$ also
counts, by edges,  {Eulerian \emm partial,} orientations of planar maps:
that is, only \emm some, edges are oriented, with the condition that at any
vertex there are as many incoming as outgoing edges.\\
{\bf 3.} As  mentioned above, the series $\Rgf(t)$ already occurs in
the map literature, and more precisely in
the enumeration of quartic maps $M$ weighted by their Tutte polynomial $\Tpol_M(
0,1)$. However, our proof does not rely on this observation, and it
remains an open problem to understand this connection combinatorially. We refer
to the final section for more details.

\medskip
The counterpart of Theorem~\ref{thm:4} for all {rooted} planar Eulerian orientations
reads as follows.

\begin{Theorem}\label{thm:gen}
  Let $\Rgf(t)\equiv \Rgf$ be the unique formal power series with constant
  term $0$ satisfying
\[
t= \sum_{n \ge 0} \frac 1 {n+1} {2n \choose n}^2\Rgf^{n+1}.
\]
Then the \gf\ of  {rooted} planar Eulerian orientations, counted by
edges, is
\[
\Ggf(t)= \frac 1{4t^2}\left( t-2t^2-\Rgf(t)\right).
\]
This is a differentially algebraic series, satisfying a non-linear
differential equation of order $2$ whose coefficients are polynomials in $t$. The number $g_n$ of such
orientations having $n$ vertices behaves asymptotically as 
\[
g_n \sim \kappa \, \frac{\mu^{n+2} }{n^2 (\log n)^2}
\]
where
\[ \kappa =1/16  \qquad \hbox{and } \qquad \mu= 4 \pi .
\]
The series $\Ggf(t)$ is not D-finite.
\end{Theorem}
The first coefficients of $\Rgf$ and $\Ggf$ are
$$
\Rgf(t)= t-2t^2-4t^3-20t^4-132t^5 -\cdots, \qquad 
\Ggf(t)=t+5t^2+33t^3+ \cdots.
$$

\medskip
\noindent{\bf Remark.} In Section~\ref{sec:bij} {we will prove} that
$2\Ggf(t)$ {also has} an interpretation {in} terms of 3-coloured maps: it
counts (by faces) properly 3-coloured quadrangulations having no
bicolored face. Equivalently, it counts Eulerian orientations of
quartic maps with no \emm alternating vertex, (a vertex where the
order of the edges would be in/out/in/out). This is the special
case $\alpha=\beta$ of a two matrix model studied in~\cite{kazakov-zinn-justin}, where
the point $\alpha=\beta=1/(4\pi)$ is indeed identified as critical. 

\medskip

\noindent {\bf Outline of the paper.} In Section~\ref{sec:def} we begin
  with basic definitions on maps, orientations, and generating
  functions. We also discuss various models related to quartic
  Eulerian orientations. In Section~\ref{sec:func4} we write a system of
  functional equations that defines the \gf\ of quartic Eulerian
  orientations. We solve it in Section~\ref{sec:sol4}, using a
  guess-and-check approach. Then comes a bijective \emm intermezzo, in
  Section~\ref{sec:bij}, where we describe a bijection of
    Ambj\o rn and Budd~\cite{ambjorn-budd}. A specialization of this
  bijection   implies that
  general Eulerian orientations with $n$ edges are in one-to-two
  correspondence with  {certain restricted} quartic Eulerian orientations with
  $n$ vertices. In Section~\ref{sec:func-gen} we give a
  system of equations for these orientations, which we solve in
  Section~\ref{sec:sol-gen}. In Section~\ref{sec:asympt} we briefly discuss
  the nature of our \gfs\ and their singular behaviour, thus
  proving the  asymptotic statements in Theorems~\ref{thm:4}
  and~\ref{thm:gen}. Section~\ref{sec:final} finally raises some open problems.

\section{Definitions}
\label{sec:def}
\subsection{Planar maps}
%
A \emph{planar map} is a proper
 embedding of a connected planar graph in the
oriented sphere, considered up to orientation preserving
homeomorphism. Loops and multiple edges are allowed
(Figure~\ref{fig:example-map}). The \emph{faces} of a map are the
connected components of  its complement. The numbers of
vertices, edges and faces of a planar map $M$, denoted by $\vv(M)$,
$\ee(M)$ and $\ff(M)$,  are related by Euler's relation
$\vv(M)+\ff(M)=\ee(M)+2$.
 The \emph{degree} of a vertex or face is the number
of edges incident to it, counted with multiplicity. A \emph{corner} is
a sector delimited by two consecutive edges around a vertex;
hence  a vertex or face of degree $k$ {is incident to}  $k$ corners. 
The \emph{dual} of a
map $M$, denoted $M^*$, is the map obtained by placing a 
vertex of $M^*$ in each face of $M$ and an edge of $M^*$ across each
edge of $M$; see Figure~\ref{fig:example-map}, right. A map is said to be
\emph{quartic} if every vertex has degree~4. Duality transforms
quartic maps  into \emm quadrangulations,, that is, maps in which every face has
degree~4. A planar map  is \emm Eulerian, if every vertex has 
even degree. Its dual, with even face degrees, is then \emm
bipartite., We call  a face
  of degree~2 (resp. 4) a \emm digon, (resp. \emm quadrangle,).

\begin{figure}[h]
  \centering
  \includegraphics{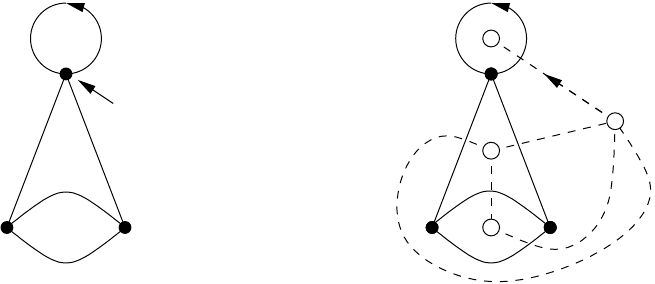}
\caption{Left: a rooted planar map, with the root edge and root corner
  shown. Right: the dual map, in dashed edges.}
\label{fig:example-map}
\end{figure}

For counting purposes it is convenient to consider \emm rooted, maps. 
A map is rooted by choosing an edge, called the root edge, and
orienting it. The starting point of this oriented edge is then the
\emm root vertex,, the other endpoint is the \emm co-root vertex,. The
face to the right of the root edge is the \emm
root face,, {and its edges are the \emm outer edges,. The face
to the left of the root edge is the \emm co-root face,.} Equivalently, one can root the map by selecting a
corner. The correspondence between these two rooting conventions is
that the oriented root edge follows the
root corner in anticlockwise 
order around the root vertex.
 In figures, we  usually choose the root face
as the infinite face (Figure~\ref{fig:example-map}). 
This explains why we often call the root face the \emm outer face, and
its degree the \emm outer degree, (denoted {$\od(M)$}). The other faces
are called \emm inner faces,. Similarly, we call the corners of the
outer face {\em outer corners} and all other corners {\em inner
  corners}.

From now on, every {map}
is \emph{planar} and \emph{rooted}, and these precisions will
  often be omitted. Our convention for rooting the dual of a map is
  illustrated on the right of Figure~\ref{fig:example-map}. Note that it makes duality of rooted
  maps a transformation of order 4 rather than 2. By convention,
we include among rooted planar maps the \emph{atomic map}
 having one vertex and no edge.

\subsection{Orientations}
\label{sec:def-orientations}
 
  A (planar) \emph{Eulerian orientation} is a (rooted, planar) map in
  which all edges are oriented, in such a way that
  the in- and out-degrees of {each} vertex are equal. We require
that the orientation chosen {for} the root edge is consistent with its
orientation coming from the rooting (Figure~\ref{fig:defs}, right).  Note that
the underlying  map must be Eulerian. We find {it} convenient to work with
duals of Eulerian orientations, which turn out to be equivalent to
certain \emm labelled maps,.

\begin{Definition}\label{def:labelled-map}
  A \emm labelled map, is a rooted planar map with integer labels on
  its vertices,  such that  adjacent labels differ by $\pm1$ and
  the root edge is labelled from $0$ to $1$. Such a map is necessarily
  bipartite. We also consider the atomic map, with a single  vertex (labelled $0$), to be a labelled map.
\end{Definition}
An example is shown in Figure~\ref{fig:labelled}.

\begin{figure}[htb]
  \centering
   \scalebox{0.9}{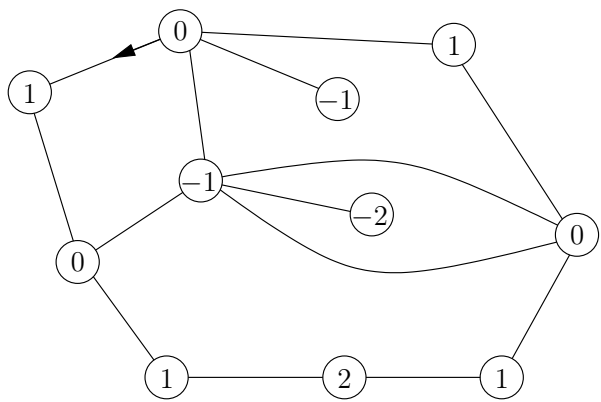}
  \caption{A labelled map.}
  \label{fig:labelled}
\end{figure}

\begin{Lemma}\label{lem:duality}
  The duality transformation between Eulerian maps and bipartite maps
  can be extended into  a bijection between Eulerian orientations
  and labelled maps, which preserves the number of edges and
  exchanges vertex degrees and face degrees.
\end{Lemma}
The construction was already used
in~\cite[Prop.~2.1]{elvey-guttmann17}. It is illustrated in
Figure~\ref{fig:duality}. The idea is that an Eulerian orientation of
edges of a map gives a height function on its faces, or equivalently,
on the vertices of its dual. Height functions on regular grids, like the  square lattice, are much studied as models of  discrete random surfaces, expected to  converge to the Gaussian free field~\cite{chandgotia,glazman-manolescu}.

\begin{figure}[htb]
  \centering
   \scalebox{0.9}{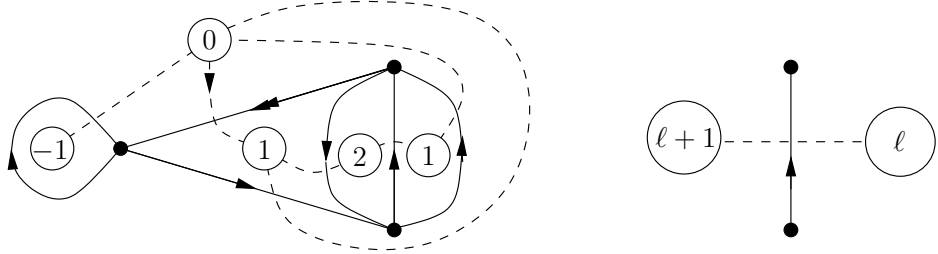}
  \caption{An Eulerian orientation (solid edges) and the corresponding
    dual labelled map (dashed edges). The labelling rule is shown on
    the right.}
  \label{fig:duality}
\end{figure}

In the  case of a quartic Eulerian orientation, the
  (quadrangular) faces of the dual map can only have two types of
  labelling, shown in Figure~\ref{fig:two_vertex_types}. It is easily shown
  that, upon replacing every label by its value modulo 3, one obtains
  a proper 3-colouring of the vertices of the quadrangulation. Conversely, given a
  3-coloured quadrangulation such that the  root edge
  is oriented from 0 to 1, all faces must be of one of the types
  shown in Figure~\ref{fig:two_vertex_types}, and one can directly reconstruct an Eulerian orientation
of the dual quartic map using the rule of  Figure~\ref{fig:duality}.
Then the
  associated labelled quadrangulation projects on the coloured
  quadrangulation modulo 3. Hence 4-valent Eulerian orientations with
  $n$ vertices are in bijection with 3-coloured quadrangulations with
  $n$ faces (with the root edge oriented from 0 to 1), as claimed
  below Theorem~\ref{thm:4}. This correspondence between Eulerian
  orientations of a planar quartic graph and 3-colourings of its dual has
  been known for a long time. In the more general, non-planar case, the number of
  Eulerian orientations of a 4-valent graph is  given by the value of its
  \emm Tutte polynomial,  at the point $(0,-2)$, with no interpretation in terms of
  colourings~\cite[Sec.~3.6]{welsh-book}.

\begin{figure}[ht]
\setlength{\captionindent}{0pt}
   \includegraphics[scale=0.7]{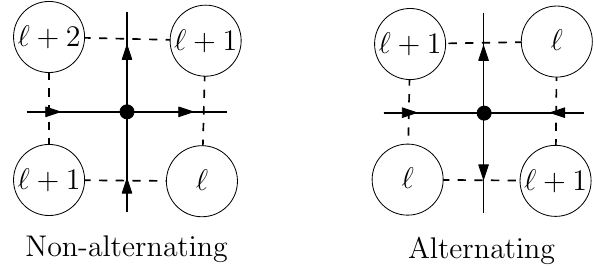} 
   \caption{The two types of vertices in a quartic Eulerian
     orientation, shown with the associated quadrangles in the dual labelled map. In the six vertex model, 
 alternating vertices are assigned the weight $\om$.}
   \label{fig:two_vertex_types}
\end{figure}

\medskip
\noindent
{\bf More orientations.} Obviously, quartic Eulerian
  orientations are orientations of a quartic map with exactly 2
  outgoing edges at each vertex. It turns out that the number of
  oriented \emm   quadrangulations, in which each vertex has outdegree
  2 is known. The associated series is D-finite. A simple bijection
  transforms these orientations into \emm bipolar, orientations
  of planar maps (no cycle, one source, one sink, both on the outer
  face)~\cite{felsner-baxter}. We refer the reader to~\cite{baxter-dichromatic,bonichon-mbm-fusy,felsner-baxter,fusy-poulalhon-schaeffer-bip}, and references
  therein. Analogous results exist for orientations of triangulations
  in which every vertex has outdegree 3, called \emm Schnyder orientations,~\cite{bernardi-bonichon,Boni05}. Let
  us also mention recent {progress regarding} bipolar orientations with
  prescribed face degrees~\cite{mbm-fusy-raschel}.

\subsection{The 6-vertex model and
  fully packed loops}
\label{sec:ice}
 The enumeration  of quartic Eulerian orientations has already been
  considered, and in some sense solved, in the mathematical physics literature, where it is
  called the ice model on a {4-valent} random lattice~\cite{kostov,zinn-justin-6V-random}. In this model an oxygen atom
    stands at every vertex, while the hydrogen atoms (two per oxygen  in  a water/ice molecule)
    lie on the edges, the arrows indicating with which oxygen
    they go. This is a special case of the \emm six vertex model,: in that
  model, the configurations are still Eulerian orientations, but a
  weight $\om=2\cos(\lambda\pi/2)$ is assigned to each vertex from
  which the two outgoing edges are opposite each other.  We call these vertices
  {\em alternating}  (see Figure \ref{fig:two_vertex_types}, right). The ice
  model then corresponds to $\om=1$, or equivalently,
  $\lambda=2/3$. In combinatorial terms, solving the six vertex model on a random lattice is equivalent to determining the refined generating $\Qgf(t,\omega)$ of
 quartic Eulerian orientations, {where $t$ still counts
   vertices}, and a weight $\om$ is assigned to
alternating vertices. Figure~\ref{fig:two_vertex_types} shows that
 $\Qgf(t,\om)$ also counts labelled quadrangulations by faces, with a
 weight $\om$ per face having  only two distinct labels.

 Kostov exactly solved the problem for general $\lambda$, though his solution was not entirely rigorous~\cite{kostov}. Kostov's solution relied on analysing the limiting eigenvalue distribution of a sequence of matrices, using results from complex analysis to determine this distribution. We had initially overlooked this solution, in part due to the unfamiliar language and techniques used. 
 In a forthcoming paper, the second author and Zinn-Justin provide a
 rigorous version of Kostov's derivation, while also fixing a mistake,
 resulting in a much simplified formula for $\Qgf(t,\omega)$ compared
 to the (incorrect) formula that one could extract directly
 from~\cite{kostov}. This new formula is written
 parametrically in terms of Jacobi theta functions (see~\cite{mbm-aep-zinn} for an extended abstract). In another forthcoming paper, the current authors generalise the methods in the present paper to rederive the same formula for $\Qgf(t,\omega)$, with our new derivation staying (almost) entirely within the world of formal power series. Our exact formula for $\Qgf(t,1)$ can also be shown to agree with the general formula for $\Qgf(t,\omega)$ at $\om=1$, though the equivalence is not obvious~\cite{mbm-aep-zinn}.

The quantity studied by Kostov is not exactly the series $\Qgf(t, \omega)$, but  the \emm free energy, $Z(t,\om)$ related to $\Qgf$ by
  \[
    \Qgf(t,\om)=2t\frac{d}{dt}\log(Z(t,\om)).
  \]
 Kostov,
  and also  Zinn-Justin~\cite{zinn-justin-6V-random}, predicted that for
  $\lambda\in[0,2)$, {that is, $\om \in (-2,2]$}, the {dominant}
    singularity of $Z(t, \om)$
occurs at
  \beq\label{rho}
\rho =
{\frac{1}{8\lambda\pi}\frac{\sin(\lambda\pi/4)}{\cos(\lambda\pi/4)^3}.}
\eeq
  In terms of $\om$, this is
  \[
\rho=\frac{1}{4\arccos(\om/2)}\frac{\sqrt{2-\om}}{(2+\om)^{3/2}}.\]
{Moreover}, Kostov \cite{kostov} predicted that the behaviour of the free energy around this singularity is
  \[\log(Z(t,\om))\sim \frac{(1-t/\rho)^2}{\log(1-t/\rho)},\]
{up to some multiplicative constant}.
This would result in:
\[
\Qgf(t, \om) \sim \frac{1-t/\rho}{\log(1-t/\rho)},
\]
up to some multiplicative constant. 

The generating function $\Qgf(t)$ of Theorem~\ref{thm:4} is
$\Qgf(t,1)$, so the predictions of Zinn-Justin and Kostov at $\om=1$
are verified by Theorem \ref{thm:4}
(see~Proposition~\ref{prop:asymptR} for the singular behaviour
  of $\Qgf(t)$, from which the asymptotic behaviour of the numbers
  $q_n$ stems). We also claim that our second theorem, Theorem~\ref{thm:gen}, solves the
case $\om=0$ 
 of the six vertex model. Indeed,  we will show that the generating function
$\Ggf(t)$ of general Eulerian orientations satisfies
$2\Ggf(t)=\Qgf(t,0)$ (see  Corollary
\ref{cor:colourful}). Hence the predictions of Zinn-Justin and Kostov for
$\om=0$ follow from Theorem \ref{thm:gen} and~Proposition~\ref{prop:asympt-gen}.

In our forthcoming paper, we analyse the exact formula for $\Qgf(t,\om)$. 
Our analysis strongly suggests 
that  the prediction~\eqref{rho} does not  hold
on the entire segment $(-2,2]$, but only for $\om >\om_c$, where
$\om_c$ is around $-0.76$.

\bigskip
\noindent{\bf Fully packed loops}\\
The case $\om=2$ is also well-understood, and boils down to
counting all planar maps weighted by their \emm Tutte polynomial,,
evaluated at the point $(3,3)$. This can be justified as follows:
starting from a quartic Eulerian orientation, we first transform each vertex
into a pair of vertices with degree~3, as shown in Figure
\ref{fig:two_vertex_types_transformations}. (This transformation has
already been used in, e.g.,~\cite{kostov,zinn-justin-6V-random}). The two possible choices
for each {alternating vertex} account for the weight $\om=2$ assigned to
these vertices. The resulting map is a cubic map in which each
vertex has one incoming edge, one outgoing edge and one
undirected edge. This transformation can be reversed by simply
contracting all undirected edges in the cubic map. The oriented edges
must then form loops on the cubic map, where each loop is oriented one
of two ways --- either clockwise or anticlockwise. Moreover,
  every vertex must be visited by a loop. In~\cite[Sec.~2.1]{borot3}, this model of
fully packed loops on cubic maps is shown to be equivalent to the
4-state Potts model on general planar maps, in which every
monochromatic edge gets a weight 3, and every vertex a weight
$1/2$. {Finally, using} the correspondence between the Potts model and
the Tutte polynomial (see, e.g.,~\cite[Sec.~3.3]{bernardi-mbm-alg}), we conclude that
\[
\Qgf(t,2)= \sum_{M \small{\hbox{ planar}}} t^{\ee(M)} \Tpol_M(3,3)
=6\,t+78\,{t}^{2}+1326\,{t}^{3}+25992\,{t}^{4}+O ( {t}^{5}) ,
\]
where $\Tpol_M(\mu, \nu)$ is the Tutte polynomial of $M$
(see~\cite{welsh-book}). This series was proved to satisfy an (explicit) non-linear
differential equation of order 3 (see~\cite[Thm.~16]{BeBM-15} for $\beta=2$), but, to our knowledge,
the singular behaviour of $\Qgf(t,2)$ has not been derived from
it. From this differential equation, one can in fact guess-and-prove a smaller differential
  equation for $\Qgf(t,2)$, of order 2 
  and degree $2$ (for $\om=0$ or $1$, we obtain DEs of order 2
    but degree 3). Written in terms of the series $S(t)=t^2(1+\Qgf(t,2))$ of~\cite{BeBM-15}, it reads:
  \[
    (1-32t)\left( 6S-2tS'-t\right)S''+2t(1-4S')^2=0.
  \]

\begin{figure}[t]
\setlength{\captionindent}{0pt}
   \includegraphics[scale=0.7]
{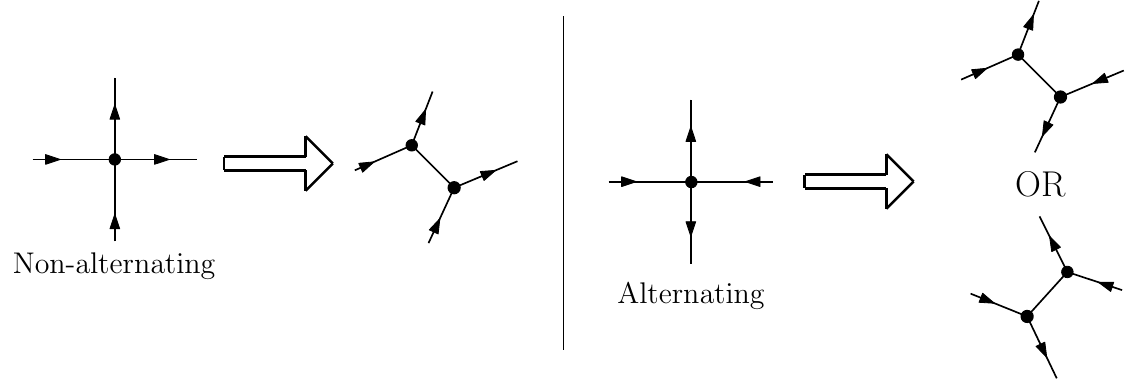} 
   \caption{The transformation from vertices in a quartic Eulerian orientation to pairs of vertices in certain cubic Eulerian partial orientations. This transformation applies when  $\om=2$.}   \label{fig:two_vertex_types_transformations}
\end{figure}

\subsection{Formal power series}

Let $A$ be a commutative ring and $x$ an indeterminate. We denote by
$A[x]$ (resp. $A[[x]]$) the ring of polynomials (resp. \fps) in $x$
with coefficients in $A$. If $A$ is a field,  then $A(x)$ denotes the field
of rational functions in $x$. We will also consider  Laurent series in $x$, that is, series of the form
\[ \sum_{n \ge n_0} a_n x^n,
\]
with $n_0\in \zs$ and $a_n\in A$.  The coefficient of $x^n$ in a   series $F(x)$ is denoted by
$[x^n]F(x)$. 

This notation is generalised to polynomials, fractions
and series in several indeterminates. For instance, the  \gf\ of
Eulerian orientations, counted by edges (variable $t$) and faces
(variable $z$)  belongs to
$\qs[[t,z]]$. For a multivariate series, say $F(x,y) \in \qs[[x,y]]$,
the notation $[x^i]F(x,y)$ stands for the \emm series, $F_i(y)$ such
that $F(x,y)=\sum_j F_j(y) x^j$. It should not be mixed up with the
coefficient of $x^iy^0$ in $F(x,y)$, which we denote by $[x^i
y^0]F(x,y)$.
If $F(x,x_1, \ldots, x_d)$ is a series in the $x_i$'s whose
coefficients are Laurent series in $x$, say
\[
F(x,x_1, \ldots, x_d)= \sum_{i_1, \ldots, i_d} x_1^{i_1} \cdots
x_d^{id}
\sum_{n \ge n_0(i_1, \ldots, i_d)} a(n, i_1, \ldots, i_d) x^n,
\]
then we define the \emm non-negative part of $F$ in $x$, as the
following \fps\ in $x, x_1, \ldots, x_d$:
\[
[x^{\ge 0}]F(x,x_1, \ldots, x_d)= \sum_{i_1, \ldots, i_d} x_1^{i_1} \cdots
x_d^{id}
\sum_{n \ge 0} a(n, i_1, \ldots, i_d) x^n.
\]
We define similarly the \emm positive part, of $F$ in $x$, denoted
$[x^{>0}]F$.

If $A$ is a field,  a power series $F(x) \in A[[x]]$
  is \emm algebraic, (over $A(x)$) if it satisfies a
non-trivial polynomial equation $P(x, F(x))=0$ with coefficients in
$A$. It is \emm differentially algebraic, (or \emm D-algebraic,) if it satisfies a non-trivial polynomial
differential equation $P(x, F(x), F'(x), \ldots, F^{(k)}(x))=0$ with
coefficients in $A$. It is \emm D-finite, if it satisfies a \emm
linear, differential equation with coefficients  in $A(x)$. For
multivariate series, D-finiteness and D-algebraicity require the
existence of a differential equation \emm in each variable,.  {We refer to~\cite{Li88,lipshitz-df} for
general results on D-finite series, and to~\cite[Sec.~6.1]{BeBMRa17} for
 D-algebraic series.}


\section{Functional equations for quartic Eulerian orientations}
\label{sec:func4}

In this section we will characterise  the generating function $\Qgf(t)$
of labelled quadrangulations by a system of functional equations.
\begin{Theorem}\label{thm:systemQ}
  There exists a unique $3$-tuple of series, denoted $\Pgf(t,y)$,
  $\Cgf(t,x,y)$ and $\Dgf(t,x,y)$, belonging respectively to $\qs[[y,t]]$,
  $\qs[x][[y,t]]$ and $\qs[[x,y,t]]$, and satisfying the
  following equations:
\begin{align*}
\Pgf(t,y)&=\frac{1}{y}[x^1]\Cgf(t,x,y),\\
\Dgf(t,x,y)&=\frac{1}{1-\Cgf\left(t,\frac{1}{1-x},y\right)},\\
\Dgf(t,x,y)&=1+{y}\, [x^{\geq0}]\left(\Dgf(t,x,y)\left([y^{1}]\Dgf(t,x,y)+\frac{1}{x}\Pgf\left(t,\frac{t}{x}\right)\right)\right),
\end{align*}
together with the initial condition
\[[y^{1}]\Dgf(t,x,y)=\frac{1}{1-x}\left(1+2t[y^2]\Dgf(t,x,y)-t([y^1]\Dgf(t,x,y))^2\right).\]
The \gf\ $\Qgf(t)$ that counts labelled quadrangulations by faces
is
\[
\Qgf(t)=[y^1]\Pgf(t,y)-1.
\]
By  Lemma~\ref{lem:duality},  the series $\Qgf(t)$  also counts  quartic Eulerian orientations  by
vertices.
\end{Theorem}

\medskip
\noindent{\bf Remarks}\\
{\bf 1.} With the conditions {on} the series
$\Pgf$, $\Cgf$ and $\Dgf$, the operations that occur in the above
equations are always well defined:
\begin{itemize}
\item 
the coefficient of $x$ in $\Cgf(t,x,y)$ lies  in $\qs[[y,t]]$,
\item the series $\Cgf(t, 1/(1-x),y)$  {indeed lies in} $\qs[[x,y,t]]$ (upon
expanding the powers of $1/(1-x)$ as series in $x$),
\item denoting by $p_{j,n}\in \qs$ the coefficient of $y^j t^n$ in
$\Pgf(t,y)$, and by $d_{j,n}(x)\in \qs[[x]]$ the coefficient of $y^j t^n$ in
$\Dgf$, the quantity
\begin{align*}
 \frac 1 x \Dgf(t,x,y) \, \Pgf\!\left(t, \frac tx \right)
&=  \left(\sum_{j,n\ge 0} d_{j,n}(x) y^j t^n
  \right)\left( \sum_{i, m \ge 0} p_{i,m}  \frac 1{x^{i+1 }}t^{i+m}\right)
\\
&= \sum_{j, N\ge 0} y^j t^N \sum_{i+m+n=N} p_{i,m}\,d_{j,n}(x) \frac 1{x^{i+1 }}
\end{align*}
is a series in $y$ and $t$ whose coefficients are Laurent series in
$x$ (because $i, m$ and $n$ are bounded). It thus makes sense to
extract its non-negative part in $x$, which will lie in $\qs[[x,y,t]]$.
\end{itemize}
{\bf 2.} In~\cite{elvey-guttmann17}, another system was given to
characterise the series $\Qgf(t)$. It is more complicated than the one
above. In particular, it involves three additional variables (other
than the main size variable $t$) rather than two. We could not solve
that complicated system, but we solve the one above in the next
section.

\medskip
The series  $\Cgf$, $\Dgf$ and $\Pgf$ of Theorem~\ref{thm:systemQ}
count certain labelled maps, which we now define. See
Figure~\ref{fig:DandP_examples} for an illustration.

\begin{Definition}\label{def:patches}
A {\em patch} is a labelled map in which each inner face has degree $4$,
and the vertices around the outer face are alternately labelled $0$ and
$1$. 

A \emm C-patch, is a patch satisfying two additional
  conditions:  all neighbours of the root vertex are labelled
  $1$, and the root corner is the only outer corner at the root
  vertex. By convention, the atomic patch is \emm not, a C-patch.

 {\em D-patches} resemble patches but may include digons.  More
precisely, a D-patch is a labelled map in which each inner face 
has degree $2$ or $4$, those of degree $2$ being incident to
the root vertex, and the vertices around the outer face
are alternately labelled $0$ and $1$. We also require that all
neighbours of the root vertex are labelled $1$.
\end{Definition}

\begin{figure}[ht]
\setlength{\captionindent}{0pt}
   \includegraphics[scale=0.7]{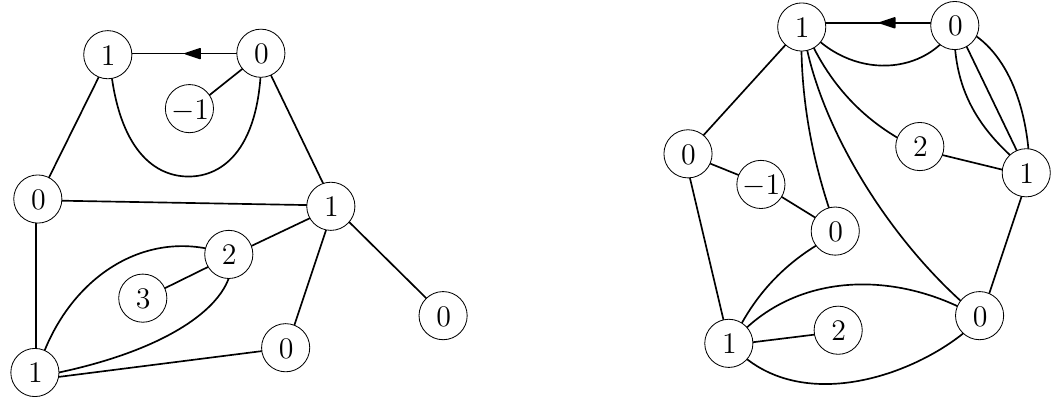} 
   \caption{
Left: a patch which contributes $t^{5}y^{4}$ to the generating
function $\Pgf(t,y)$. It does not satisfy  the first condition
  of  a C-patch. Right: a D-patch which contributes
  $t^{6}x^{3}y^{3}$ to the generating function
  $\Dgf(t,x,y)$.}
   \label{fig:DandP_examples}
\end{figure}

We define $\Pgf(t,y)$, $\Cgf(t,x,y)$ and $\Dgf(t,x,y)$ to be
respectively the \gfs\ of patches, C-patches and D-patches, where $t$
counts inner quadrangles, $y$ the outer degree (halved), and $x$
either the degree of the root vertex (for C-patches)  or the number of
inner digons (for D-patches). Comparing with the previous
  paper giving functional equations for this problem~\cite{elvey-guttmann17}, we see that one
  parameter, namely the degree of the co-root vertex, is no longer
  involved here.
The series $\Pgf$, $\Cgf$ and $\Dgf$ actually
  belong to the rings prescribed by Theorem~\ref{thm:systemQ}: 
  \begin{itemize}
  \item for  {$\Pgf$}  it suffices to observe that there are finitely
  many patches with  $n$ inner quadrangles and outer degree $2j$,
\item for {$\Dgf$}  we observe that there are finitely
  many D-patches with $n$ inner quadrangles, $i$ inner digons and outer
  degree $2j$, 
\item  finally for {$\Cgf$}, we note that a C-patch with $n$ inner
  quadrangles cannot have  a root
  vertex of degree  larger than $1+4n$, because all non-root corners
  at the root vertex must belong to an inner quadrangle (by the second
  condition of Definition~\ref{def:patches}). This explains
  the polynomiality of $[t^n]\Cgf$ in $x$ (and yields in fact a smaller ring than
  $\qs[x][[y,t]]$, namely $\left(\qs[[y]][x]\right)[[t]]$, but this
  won't be needed).
  \end{itemize}

In the next 5 lemmas, we prove that the series that we have defined
satisfy the 5 equations of Theorem~\ref{thm:systemQ}. We will finish the
section by proving that the system has a unique solution in the prescribed rings of series.

\begin{figure}[ht]
\setlength{\captionindent}{0pt}
   \includegraphics[scale=0.7]{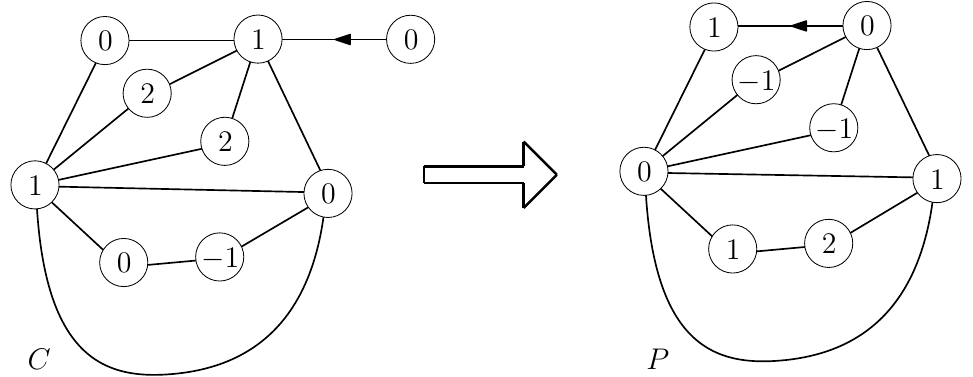} 
   \caption{The transformation of $C$ into $P$ used in the proof of
     Lemma~\ref{PfromC}.}
   \label{fig:CfromP_example}
\end{figure}

\begin{Lemma}\label{PfromC}
The generating functions $\Pgf(t,y)$ and $\Cgf(t,x,y)$ satisfy the equation
\[\Pgf(t,y)=\frac{1}{y}[x^1]\Cgf(t,x,y).\]
\end{Lemma}
\begin{proof}
Let $C$ be any C-patch  counted by $[x^{1}]\Cgf(t,x,y)$, that is, in
which the root vertex has degree 1.  We construct a new patch $P$ from
$C$, as illustrated in Figure~\ref{fig:CfromP_example}: we delete the root edge and root
vertex of $C$,   replace each label $\ell$ with $1-\ell$, and finally
root $P$ at the outer edge  of~$C$ following the root edge of $C$
anticlockwise. Then the new labelled map $P$ is indeed a patch. If $C$
 contains only one edge then $P$ is the atomic map. The outer degree
 has decreased by 2, while the number of inner quadrangles is
 unchanged. Finally, the transformation from $C$ to $P$ is
 reversible. This proves the lemma.
\end{proof}

\begin{figure}[ht]
\setlength{\captionindent}{0pt}
   \includegraphics[scale=0.7]{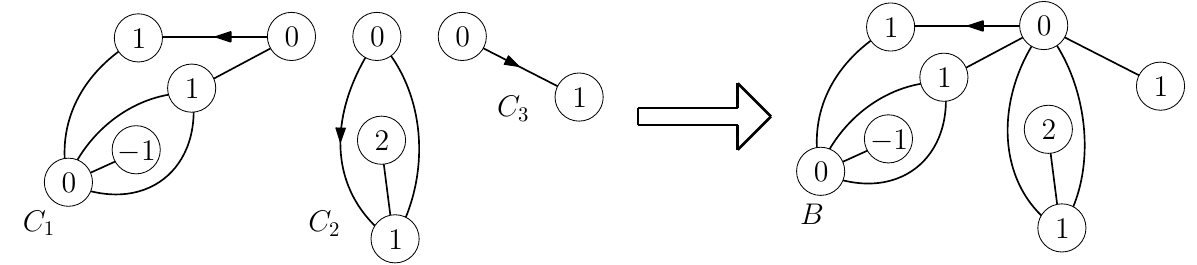} 
   \caption{A sequence of C-patches gives rise to a B-patch, as in
     Lemma~\ref{DfromC}.}
   \label{fig:DfromC_example}
\end{figure}

\begin{Lemma}\label{DfromC}
The generating functions $\Dgf(t,x,y)$ and $\Cgf(t,x,y)$ satisfy the equation
\[
  \Dgf(t,x,y)=\frac{1}{1-\Cgf\left(t,\frac{1}{1-x},y\right)}.
\]
\end{Lemma}
\begin{proof}
Recall that C-patches satisfy two conditions: all neighbours of
the root vertex have label~1, and the root vertex is only incident
once to the root face. By attaching a sequence of C-patches at their
root vertex, as shown in Figure~\ref{fig:DfromC_example}, we form a
\emm B-patch,, that is, a patch satisfying
only the first of these conditions. The associated \gf\ is 
\[
\Bgf(t,x,y)=\frac{1}{1-\Cgf(t,x,y)}.
\]
As before, $t$ counts inner quadrangles, $x$ the degree of the root
vertex and $y$ the outer degree (halved). 

\begin{figure}[ht]
\setlength{\captionindent}{0pt}
   \includegraphics[scale=0.7]{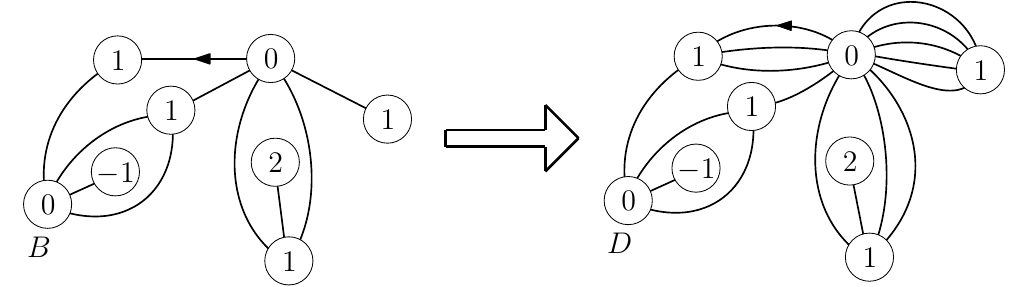} 
   \caption{The transformation from a B-patch to a D-patch, as in
     Lemma~\ref{DfromC}.}   \label{fig:DfromB_example}
\end{figure} 

Now in order to construct a  D-patch, {it suffices to take a B-patch and inflate every edge which is incident to the root vertex} into a sequence of digons, as shown in Figure \ref{fig:DfromB_example}. This explains the transformation $x\mapsto 1/(1-x)$ occurring in the lemma. In this way
the variable $x$ now counts digons of D-patches.
\end{proof}

\begin{figure}[b]
\setlength{\captionindent}{0pt}
   \includegraphics[scale=0.7]{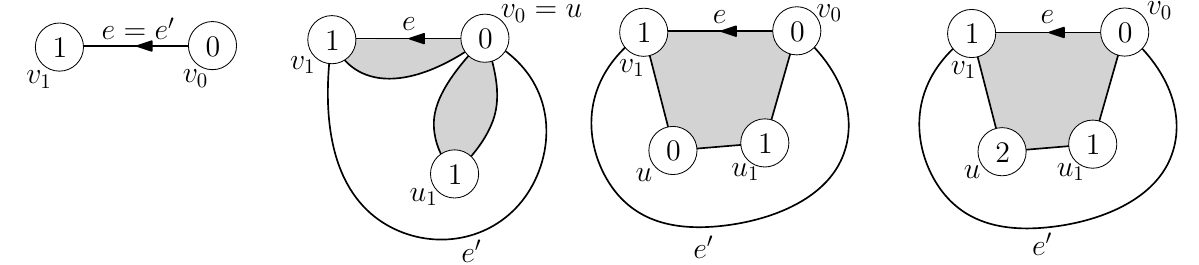} 
   \caption{The four different types of patches which contribute to
     the generating function $[y^{1}]\Cgf(t,x,y)$. 
In the third and fourth cases it is possible that
$u_{1}=v_{1}$. The shaded area represents a labelled quadrangulation.}
   \label{fig:Dout2_cases}
\end{figure}

\begin{Lemma}\label{Dout2}
The generating function $\Dgf(t,x,y)$ satisfies the equation
\[[y^{1}]\Dgf(t,x,y)=\frac{1}{1-x}\left(1+2t[y^2]\Dgf(t,x,y)-t([y^1]\Dgf(t,x,y))^2\right).\]
\end{Lemma}
\begin{proof}
We will show that
\beq\label{eq:C1}
[y^{1}]\Cgf(t,x,y)=x\left(1+2t[y^2]\Cgf(t,x,y)+t([y^1]\Cgf(t,x,y))^2\right),
\eeq
from which the desired result follows using Lemma~\ref{DfromC}, while observing that $\Cgf(t,x,0)=0$.
Let~$C$ be any C-patch counted {by} $[y^{1}]\Cgf(t,x,y)$, that is, having
outer degree 2. Let $e$ be the root edge of $C$, let $e'$ be the other
outer edge of $C$ and let $v_{0}$ and $v_{1}$ be the root vertex and
co-root vertex respectively. We  consider four cases, illustrated
in Figure~\ref{fig:Dout2_cases}. 

In the first
case $e=e'$. Since the outer degree of $C$ is 2, this is only possible
if $e$ is the only edge in $C$, so this case simply contributes $x$ to
$[y^{1}]\Cgf(t,x,y)$. 
For the other three cases, let $Q$ be the map remaining when $e'$ is
removed (this is the shaded area in
  Figure~\ref{fig:Dout2_cases}). Then the outer degree of $Q$ must be
4, so that $Q$ is a quadrangulation. Let the vertices
around the outer face of~$Q$ be $v_{0}$, $v_{1}$, $u$ and $u_{1}$ in
anticlockwise order. Note that $v_{1}$ and $u_{1}$ must both be
labelled 1 since they are adjacent to $v_{0}$ and $C$ is a C-patch. 

The second case we consider is when $u=v_{0}$. Then $Q$ can
be separated into two C-patches with outer degree~2,
hence this case contributes 
\[xt([y^1]\Cgf(t,x,y))^2\]
to $[y^1]\Cgf(t,x,y)$. The factor $xt$ appears because  the number of inner quadrangles in $Q$ and the degree of
the root vertex of $Q$ are each one less than the equivalent numbers in $C$.

 The third case  is when $u\neq v_{0}$, but $u$ is
 labelled 0. {Then}  $Q$ can be any C-patch with outer degree 4.
Hence this case contributes
\[xt[y^2]\Cgf(t,x,y).\]

In the fourth and final case, $u$ is labelled $2$, and therefore it
cannot be equal to $v_{0}$. In this case $Q$ is not a patch
  because of this label $2$ on its outer face. But we construct a new map $Q'$
from $Q$ by replacing every label $\ell$ in $Q$ with $2-\ell$, except
for the label at the root vertex, which remains 0. Then $Q'$ is still
a labelled map,  all neighbours of the root vertex are
still labelled 1, and the root face is only incident once to
  the root vertex.
Hence, $Q'$ can be any C-patch with
outer degree~4,
so this case contributes
\[xt[y^2]\Cgf(t,x,y).\]

Adding the contributions from the four cases yields~\eqref{eq:C1},
which, in turn, yields the desired result using Lemma~\ref{DfromC}.
\end{proof}

\begin{figure}[ht]
\setlength{\captionindent}{0pt}
   \includegraphics[scale=0.7]{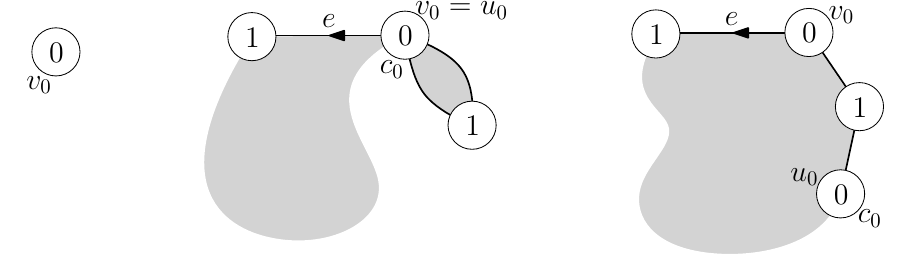} 
   \caption{The three different types of D-patches.
In the third case it is possible that the two displayed vertices labelled 1 are the same vertex.}
   \label{fig:Dproof_cases}
\end{figure}

In order to prove the most complex equation of our system,
\beq\label{third}
  \Dgf(t,x,y)=1+y\,
  [x^{\geq0}]\left(\Dgf(t,x,y)\left([y^{1}]\Dgf(t,x,y)+ \frac{1}{x}\Pgf\!\left(t,\frac{t}{x}\right)\right)\right),
\eeq
we will consider three  types of {D-patches},
illustrated in Figure~\ref{fig:Dproof_cases}, and we will enumerate the D-patches of each type separately. The first type  is just the atomic map, 
which contributes $1$ to $\Dgf(t,x,y)$. For any other D-patch $D$, let
$v_{0}$ be the root vertex, let 
$c_0$ be the outer  corner labelled $0$ that follows the root
corner clockwise   around the outer face, and let
$u_{0}$ be the vertex associated with $c_{0}$. We define
D-patches of type 2 as those that
satisfy $u_{0}=v_{0}$, while  D-patches of type 3 satisfy $u_{0}\neq v_{0}$.
\begin{Lemma} The contribution to $D(t,x,y)$ from  D-patches of type
  $2$ is given by
\[
y\left([y^1]\Dgf(t,x,y)\right)\Dgf(t,x,y).
\]\end{Lemma}
\begin{proof}The result follows from the fact that any  D-patch of
  type $2$
can be split into two D-patches at $v_{0}$ where one
has outer degree 2 and the other can be any D-patch.
\end{proof}
Note that this contribution can be written $y[x^{\ge0}]
  \left(\Dgf(t,x,y) [y^1]\,\Dgf(t,x,y)\right)$ as in~\eqref{third}. It remains to determine the contribution from  D-patches of type
$3$. 

\begin{Proposition} \label{prop:bij3} There is a bijection  between  D-patches 
   $D$  of type $3$  and pairs $(\notem{P},D')$ of a patch $P$ and a
  D-patch $D'$ such that the number of digons in $D'$ is larger than
  half the outer degree of $P$. More precisely, if $D$ has $n$ inner
quadrangles,  $d$ inner
  digons and outer degree~$2j$, 
\begin{itemize}
\item the total number of inner faces of $P$ and $D'$ is $n+d+1$,
\item the outer degree of $D'$ is $2j-2$,
\item the number  of inner digons in $D'$ is  $d+k+1$, where $2k$ is  the outer degree of $P$.
\end{itemize}
\end{Proposition}

Before proving the proposition, let us show that it completes the proof of~\eqref{third}.

\begin{Corollary}\label{Dequation}
The contribution to $D(t,x,y)$ from  D-patches of type
  $3$ is given by
\[
  \Dgf(t,x,y)=y\,
  [x^{\geq0}]\left(\frac{1}{x}\Dgf(t,x,y) \Pgf\!\left(t,\frac{t}{x}\right)\right).
\]
\end{Corollary}
\begin{proof}
We use the bijection of Proposition~\ref{prop:bij3}, and express the statistics of $D$ in terms of those
  of $M$ and $D'$:
\begin{itemize}
\item the outer degree of $D$ is the outer degree of $D'$ plus 2,
\item the number $d$ of inner digons in $D$ is the number of inner digons in $D'$,
  minus half the outer degree of $P$, minus $1$,
\item finally, the number of inner quadrangles of $D$ is the sum of the corresponding
  numbers in $P$ and $D'$, plus half the outer degree of $P$.
\end{itemize}
Hence, the contribution from  D-patches of type 3 is 
\[
\frac{y}{x} [x^{>
  0}]\left(\Dgf(t,x,y)\,\Pgf\!\left(t,\frac{t}{x}\right)\right)
= y\, [x^{\ge
  0}]\left(\frac{1}{x}\, \Dgf(t,x,y)\,\Pgf\!\left(t,\frac{t}{x}\right)\right).
\]
\end{proof}

To prove Proposition~\ref{prop:bij3},  we need to introduce 
minus-patches, subpatches and a contraction
operation. This contraction operation was already used in
\cite{elvey-guttmann17}, on a slightly different class of patches.

\begin{Definition}\label{def:minus-patch}
  A {\em minus-patch} is a
  map obtained from  a patch by replacing each label $\ell$ with~$-\ell$.
\end{Definition}
 Clearly these are equinumerous with patches. 
We now describe a way to extract a minus-subpatch from a  D-patch of
type $3$. This
definition is illustrated on the left of
Figure~\ref{fig:subpatch_and_contraction_example}. Recall the
  notation $c_0$ for the  outer corner labelled $0$ that follows the
  root corner in  clockwise order around the outer face, and $u_0$ for the associated vertex.

\begin{figure}[ht]
\setlength{\captionindent}{0pt}
   \includegraphics[scale=0.7]
{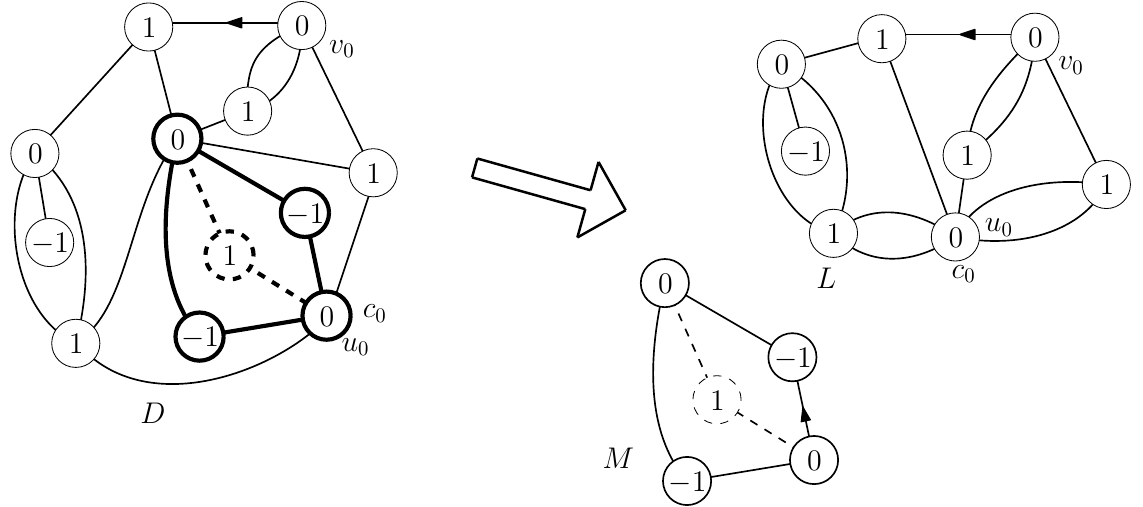} 
   \caption{Left: a  D-patch of type $3$. The minus-subpatch $M$ of
     $D$ is highlighted in     $D$ and shown separately in the
     middle. The submap $M'$ is       obtained from $M$ by deleting
     the {dashed vertex and edges}. Right: the labelled map     $L$
     constructed from $D$ by contracting $M$ to a single vertex
     $u_{0}$.} 
   \label{fig:subpatch_and_contraction_example}
\end{figure}

\begin{Definition} \label{def:minus-sub}
Let $D$ be a D-patch of type $3$. We define the {\em minus-subpatch
of $D$} as follows. First, let $M'$ be the maximal 
submap of $D$ that contains $u_{0}$ and consists of vertices
  labelled $0$ or less. Let $M$ be the submap of $D$ {that} contains $M'$
and all edges and vertices within its boundary (assuming the
  root face is drawn {as} the infinite face). The map $M$, which we root
at the corner inherited from $c_0$, 
is the {\em minus-subpatch of $D$}.
\end{Definition}
The following lemma justifies the terminology \emm minus-subpatch,.
\begin{Lemma}
The minus-subpatch of a  D-patch of type $3$ is a minus-patch.
\end{Lemma}
\begin{proof}
  In the above definition,
  it is clear that $M$ and $M'$ share
the same outer face. Moreover, all inner faces of $M$ are also inner faces
of $D$. Since the root vertex of $D$ is only adjacent to
vertices labelled 1, it cannot be a vertex of $M'$, so it
cannot be a vertex of $M$ either. Hence all inner faces of $M$ are
quadrangles, since all digons in $D$ are incident to its root. All outer vertices of $M$ must also be outer vertices of
$M'$, so they have non-positive labels. Let us prove that these labels
can only be $0$ and $-1$. For any outer vertex~$u$ of
$M$, there is some face $F$ of~$D$, containing~$u$, which is not
a face of $M$. If $F$ is
the outer face of $D$, with labels $0$ and $1$, then the label of~$u$,
being non-positive, can only be $0$.  The face $F$ cannot be a
  digon, otherwise $u$ would be the vertex labelled $0$ in this digon, and thus
  would be the
  root vertex of $D$, while we have shown that this vertex is not in
  $M$. Finally, if  $F$ is an inner quadrangle of~$D$,
 then it must contain a vertex $u'$
with label at least 1 (otherwise $F$ would be contained in $M$). Since
$u$ and~$u'$ are incident to the same quadrangle $F$, and $u$ has a
non-positive label, this  label can only be $-1$ or~$0$. Hence the outer
vertices of $M$ are all labelled $0$ or $-1$, so $M$ is a
minus-patch, {in the sense of Definition~\ref{def:minus-patch}}.
\end{proof}

The patch $P$ associated with a D-patch $D$ of type 3 in
Proposition~\ref{prop:bij3} will simply be obtained by negating the
labels in the minus-patch $M$.

  Let us now describe how $D'$ is constructed. Every edge in $D$ which connects a vertex in $M$
 to a vertex not in $M$ must have endpoints labelled 0 (in $M$) and 1
 (not in $M$). 
 We can thus contract all of $M$ to
 a single vertex  labelled~0, still denoted $u_0$, to form a new labelled map~$L$
(Figure~\ref{fig:subpatch_and_contraction_example}, right). 
The vertex $u_{0}$ is still distinct from the root vertex
$v_0$. Finally, we move $u_{0}$ towards  $v_{0}$ in the outer face of $L$
until these two vertices merge into a new root vertex $w_0$
(Figure~\ref{fig:Dproof_example}). This creates an extra inner digon at $w_0$, in addition to those that were
incident to $u_0$ and $v_0$. Note that we do not merge
  any edges. This gives a new labelled map, denoted $D'$.

\begin{figure}[ht] 
\setlength{\captionindent}{0pt}
   \includegraphics[scale=0.7]{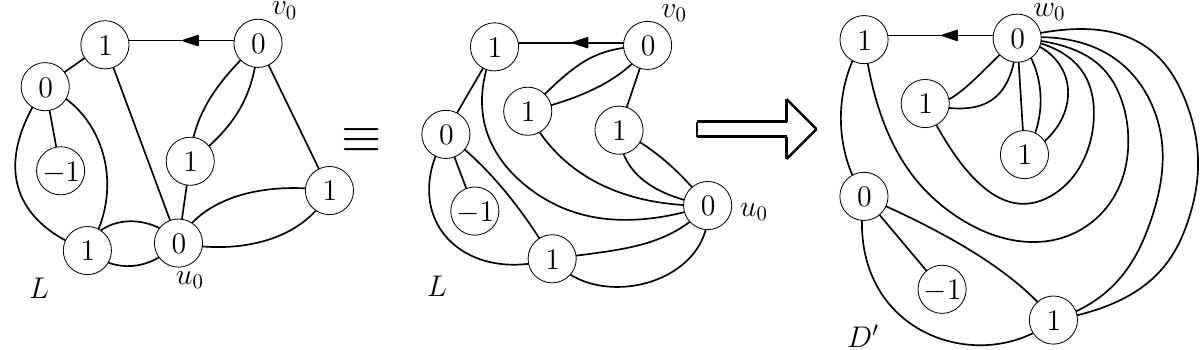} 
   \caption{The transformation of $L$ into a D-patch $D'$.}
   \label{fig:Dproof_example}
\end{figure}

\begin{Lemma}\label{lem:D}
  The labelled map $D'$ obtained by the above construction is a
  D-patch. If $D$ has  $d$ inner digons, outer degree $2j$, and its minus-patch $M$ has
  outer degree $2k$, then $D'$ has $d+k+1$ inner digons and outer
  degree $2j-2$. Finally, $D'$ and $M$ have together one more finite
  face than~$D$.
\end{Lemma}
\begin{proof}
Since $L$ is obtained by contracting the submap $M$ into a single
vertex, its inner faces cannot be bigger than the inner faces of
$D$. Hence they are quadrangles or digons. Moreover, all digons are
attached either to $v_0$ (as in $D$), or to $u_0$ (because they
result from the contraction of two edges of an inner
quadrangle). Hence, once $u_0$ and $v_0$ are merged to form the map
$D'$, all digons are incident to the new root vertex $w_0$. Finally,
all neighbours of $v_0$ in $L$ are labelled~$1$ (as in $D$), and the
same holds for all neighbours of $u_0$, because they were neighbours
of $M$, and all edges joining a vertex in $M$ to a vertex not in $M$
join label $0$ to label $1$. Hence, in $D'$, the root vertex is only
adjacent to vertices labelled $1$, and $D'$ is a D-patch.

Let us now prove the statements dealing with the statistics. Clearly, no outer edge of $D$ lies in $M$, hence
the outer degrees of $L$ and $D$ are the same. The transformation of
$L$ into $D'$ reduces the outer degree by $2$.
The statement involving the number of finite faces is also clear:
every finite face of $D$  results in a finite face of $M$ or $L$,
and transforming $L$ into $D'$ creates a new inner digon. The number
of inner digons in $D'$ is $d+k+1$, where $d$ is the number of inner
digons in~$D$, and~$k$ is the number of inner
digons attached to the vertex $u_0$ in the contracted map $L$. 
We claim that $k$ is also half the outer degree of $M$.  This comes from the fact that, from every  corner~$c$ labelled~$0$ on the
   outer face of $M$, there must start (in $D$) {at least one} edge ending at a vertex
   labelled~$1$.  Otherwise, the face of $D$ that contains $c$
   would contain two corners labelled $-1$ and would have degree larger than
   $4$, which is impossible.  Hence every pair of two consecutive
   outer edges of $M$ with labels $0, -1, 0$ occurs in a unique quadrangle,
   outside $M$, and this quadrangle will be contracted to form a digon of~$L$ adjacent
   to $u_0$.
\end{proof}

  \begin{proof}[Proof of Proposition~\ref{prop:bij3}]

    Starting from a D-patch $D$ of type 3, we construct the
    minus-patch $M$, the labelled map $L$ and the $D$-patch $D'$ as
    described above. We take for $P$ the patch obtained by negating
    labels in $M$. The statements that deal with the statistics of
    $D$, $M$ (or $P$) and $D'$ then follow from Lemma~\ref{lem:D}. We denote $f(D)=(P,D')$.

  Conversely, we need to show how to construct a D-patch $D$ of type
  $3$ from a pair $(P,D')$
  satisfying the conditions of the proposition.  Let $2k$
  and $2j-2$ be the outer degrees of $P$ and~$D'$,
  respectively, and let $d'$ be the number of digons of
  $D'$.
  Define $d:=d'-k-1$. This is
    non-negative by the assumption  $d'>k$. Split the
    root vertex $w_0$ of $D'$ into two vertices $v_0$ and $u_0$ so
    that exactly $d$ digons remain attached to $v_0$, and $k$ to
    $u_0$. This gives a labelled map $L$, in which all digons are
    attached to the root vertex $v_0$ or to $u_0$. Moreover, all
    neighbours of these vertices are labelled~$1$. The outer degree of
    $L$ is $2j$, and it has one inner face less than $D'$.

    Next, we negate the labels of the patch $P$ to obtain a minus-patch $M$. We now want to insert $M$  at the vertex $u_0$ in $L$ to construct a D-patch $D$.
  If $M$   is atomic, then we take $D=L$.  Otherwise, let
$c_{0}$ be the outer corner labelled $0$ following the root
  corner in clockwise order around the outer face in $L$. The vertex at this corner is $u_0$.
Roughly speaking, we need to place the root corner of $M$ at $c_0$, and to distribute the edges attached to $u_0$ in $L$ around the minus-patch $M$.
Let $e_1, e_2, \ldots, e_\delta$ be the
edges of $L$ attached to $u_0$, in anticlockwise order 
starting from the corner $c_{0}$
(Figure~\ref{fig:decontract}).
We now  erase the vertex $u_0$  from $L$, so that
the half-edges $e_i$ are dangling. We connect them to the 
outer corners of $M$ labelled $0$ in the following way: we first attach $e_1$ to
the root corner of $M$, and then proceed  anticlockwise around
$M$, connecting $e_{i+1}$ to the next corner of $M$ labelled
  $0$ if the corner of $L$ at $u_0$ defined by  $e_i$ and $e_{i+1}$ belongs to
  an inner digon of $L$ (this creates a new quadrangle), and to the
  same corner as $e_i$ otherwise. Recall that $M$ has outer degree $2k$, so it
  has $k$ corners labelled $0$, which is the same as the number of
  digons incident to $u_{0}$ in $L$. Hence this construction connects
  the final edge $e_{\delta}$ to the root corner of $M$, and we thus obtain a map $D$, which we
  define to be $g(P,D')$. Note also that 
  all vertices of $M$ labelled $-1$ end up on the interior of $g(P,D')$, away
  from the outer face.

\begin{figure}[htb]
  \centering
   \scalebox{0.9}{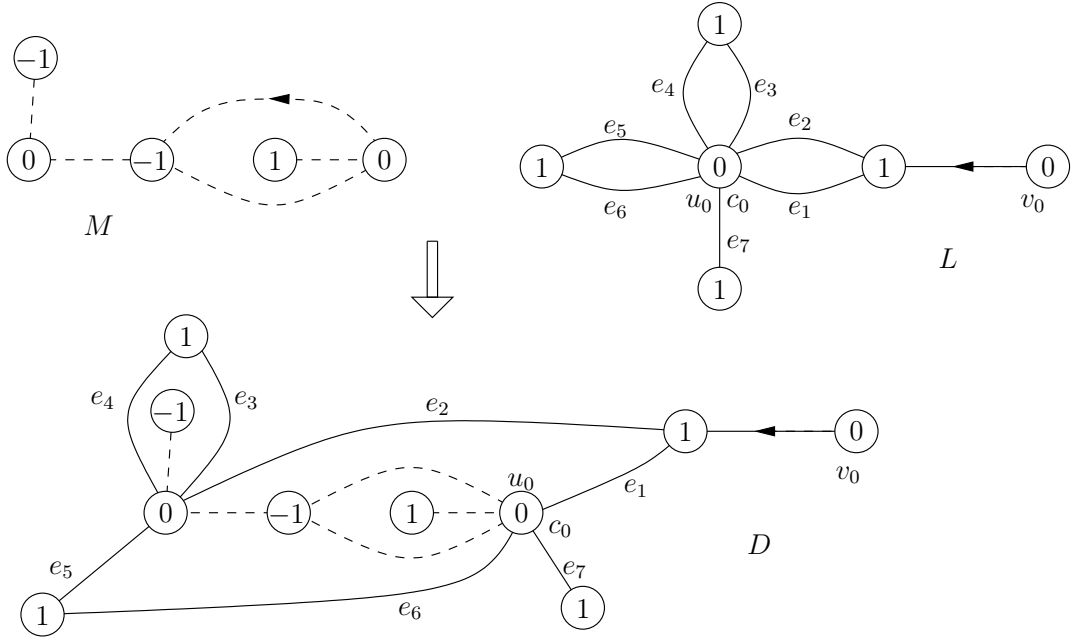}
  \caption{How to reconstruct the  D-patch $D$ from the
    minus-patch $M$ (dashed edges) and the map
    $L$. Here, $M$ has outer degree $2k=6$,  the vertex
    $u_0$ of $L$ has degree $\delta=7$ and is incident to $k=3$ digons.}
  \label{fig:decontract}
\end{figure}

  Let us explain why $D$ is a D-patch of type 3. It is clearly a labelled map, and zeroes and ones alternate on its outer face (as in $L$). When inserting $M$ in $L$, we have transformed every inner digon that was incident to $u_0$ in $L$ into an inner quadrangle: hence all inner faces of $D$ have degree 2 or 4. Finally, all neighbours of $v_0$ in $D$ are labelled $1$, as in $D'$ and $L$. Hence $D$ is a D-patch. Since we have split the vertex $w_0$ into two distinct vertices $v_0$ and $u_0$, it has type 3.

  Note that  $M$ is the minus-subpatch of $D$: indeed, it is a minus-patch, it contains $u_0$, and it is only connected to the rest of $D$ by edges labelled 0 at one end (in $M$) and 1 at the other end (out of $M$). This is the key point in proving that $f\circ g(P,D')= (P,D')$.

  Finally, to prove that $g\circ f(D)=D$ for any D-patch D of type 3,
    it suffices to observe that in the application of $g$, our choices for where to attach the edges $e_{1},e_{2},\ldots,e_{\delta}$ (Figure ~\ref{fig:decontract}) are the only choices that ensure that the resulting map is a D-patch in which $c_{0}$ is contained in the root corner of $M$. Indeed, the condition on the root corner of $M$ forces $e_{1}$ to be attached to this corner, while the rest of the choices are then forced by requirement that the inner faces of $g(P,D')$ that are incident to $u_{0}$ must be quadrangles. Hence,
  when applying $g$ to $(P,D')=f(D)$, we must obtain the map $D$.
\end{proof}

\begin{Lemma}\label{lem:Qgf}
The generating function $\Qgf(t)$ is given by 
\[\Qgf(t)=[y^1]\Pgf(t,y)-1.\]
\end{Lemma}
\begin{proof}
Let $Q$ be any labelled quadrangulation. The outer face may
  contain a label $-1$ or $2$, hence $Q$ is not necessarily a patch.  Let $P$ be the map
constructed from $Q$ by adding an edge~$e'$ between the root
vertex and co-root vertex in the outer face of $Q$, so that
$e'$ and the root edge~$e$ are the only outer edges of $P$. Then
$P$ can be any patch with outer degree 2, except for the patch with
only one edge. Hence the possible patches $P$ are counted by
$([y^1]\Pgf(t,y)-1)$. Since the number of inner faces of $P$ is equal
to the total number of faces of $Q$, this expression is exactly equal to $\Qgf(t)$. This concludes the proof.
\end{proof}

\begin{proof}[Proof of Theorem~\ref{thm:systemQ}.] We have now proved
  the five functional equations. It remains to prove that, together
  with  the conditions on the rings that contain $\Pgf$, $\Cgf$ and
  $\Dgf$, they  determine these three series.  Let us denote by $p_{j,n}$ the coefficient of $y^j t^n$ in
$\Pgf(t,y)$, and similarly for $\Cgf$ and $\Dgf$. These quantities
should be thought of respectively as elements of $\qs$ (for $\Pgf$),
of $\qs[x]$ (for $\Cgf$) and of $\qs[[x]]$ (for $\Dgf$).  We will prove by
induction on $N\ge 0$ that
\begin{itemize}
\item $p_{j,n}$ is completely determined for $j+n<N$,
\item $c_{j,n}$ and $d_{j,n}$ are completely determined for $j+n \le N$.
\end{itemize}
When $N=0$, there is nothing to prove for $\Pgf$. The third equation
of the system shows that $\Dgf-1$ is a multiple of $y$. That is, not
only $d_{0,0}=1$, but in fact we also know that $d_{0,n}=0$ for
$n\ge 1$. The
second equation then tells us that $\Cgf$ is a multiple of $y$, so
that $c_{0,n}=0$ for $n\ge 0$.
Now assume that the induction hypothesis holds for some $N\ge 0$, and let us prove it for
$N+1$. 

We begin with the series $\Dgf$. Of course it suffices to
determine the coefficients $d_{j,n}$ for $j+n=N+1$. We have already explained that
$d_{0,N+1}=0$, so we take $j\ge 1$. The third equation of the system
expresses $d_{j,N+1-j}$ in terms of the series $d_{j-1,m}$ (for $m\le
  N+1-j$), $p_{k, \ell}$ (for $k+\ell \le N+1-j$) and $d_{1,m}$ (for
$  m\le N+1-j$). If $j\ge 2$, these series are known, by the induction
hypothesis, and thus $d_{j,N+1-j}$ is completely determined. As argued
below Theorem~\ref{thm:systemQ}, it belongs to $\qs[[x]]$. To
determine the final coefficient $d_{1,N}$, we resort to the fourth
equation, which  expresses  $d_{1,N}$ in terms of $d_{2,N-1}$ (which we have
just determined) and the series $d_{1,m}$ for $m\le N-1$ (which are
known by the induction hypothesis). Again, $d_{1,N}$ belongs to $\qs[[x]]$.

Hence  for $j+n\le N+1$, the coefficients $d_{j,n}$ are uniquely
determined and hence must count D-patches with outer degree $2j$ and
$n$ quadrangles. Since we know that the \gfs\ of C-patches and
D-patches are related by the second equation (see Lemma~\ref{DfromC}), this
forces the coefficients $c_{j,n}$, for $j+n\le N+1$, to count
C-patches. Hence they are also fully determined (and are \emm polynomials, in
$x$). Finally, the first equation of the system shows that the numbers
$p_{j,n}$ are also determined for $j+n \le N$ (we cannot go up to
$N+1$ because of the division by $y$).

This concludes our induction.
\end{proof}

\section{Solution for quartic Eulerian orientations}
\label{sec:sol4}
We are now about to solve the system of Theorem~\ref{thm:systemQ},
thus proving, in particular, that the \gf\ $\Qgf(t)$ of quartic
Eulerian orientations is indeed given by Theorem~\ref{thm:4}. The
third equation of the system suggests {that we should} consider the series
$\Pgf(t,ty)$ rather than $\Pgf(t,y)$. In turn, this leads {us} to
apply the same transformation to the series $\Cgf$ and $\Dgf$. More
precisely, let us consider
\beq\label{natural-nn}
\Pnn(t,y)=t\,\Pgf(t,ty), \qquad \Cnn(t,x,y)=\Cgf(t,x,ty), \qquad
\Dnn(t,x,y)=\Dgf(t,x,ty).
\eeq
Of course, if we determine $\Pnn$, $\Cnn$ and $\Dnn$, then $\Pgf, \Cgf$
and $\Dgf$ are completely determined as well.

The solution below has been \emm guessed,, and then of course checked. The first
  step was the discovery of the connection between the \gf\ $\Qgf$ and the
  series $\Rgf$ coming from~\cite{mbm-courtiel}.  Next, writing the auxiliary series $\Pnn(t,y)$, $\Cnn(t,x,y)$ and $\Dnn(t,x,y)$ as series in $\Rgf$, $x$ and $y$, we noticed that the coefficients of $\Pnn(t,y)$ were simple products of binomial coefficients. Next, by chance we found that the series $\Dnn(t,0,1)$ appeared in the On-line Encyclopedia of Integer Sequences as the exponential of a much nicer sequence~\cite[A229452]{oeis}, so we tried taking the log of $\Dnn(t,x,y)$. We were pleasantly surprised to see that $\log(\Dnn(t,x,y))$ had very nice coefficients when written as a series in $\Rgf$, $x$ and $y$, which allowed us to guess its exact form as well as that of $\Cnn(t,x,y)$. To our knowledge,
  this is the first time that series of this form
   appear in   combinatorial enumeration.

\begin{Theorem}\label{thm:solQ}
Let $\Rgf(t)\equiv \Rgf$ be the unique formal power series with constant
  term $0$ satisfying
\beq\label{def:RQ}
t= \sum_{n \ge 0} \frac 1 {n+1} {2n \choose n}{3n \choose n} \Rgf^{n+1}.
\eeq
  Then the above series $\Pnn$, $\Cnn$ and $\Dnn$ are:
\beq \label{P-expr}
\Pnn(t,y)=\sum_{n\ge 0}\sum_{j=0}^n \frac 1{n+1} {2n-j\choose n}{3n-j
  \choose n} y^j \Rgf^{n+1},
\eeq
\[
\Cnn(t,x,y) =1 -\exp\left( 
- \sum_{n\ge 0} \sum_{j=0}^n \sum_{i=0}^{2n-j}
  \frac 1 {n+1} {2n-j \choose n} {3n-i-j\choose n}  x^{i+1}  y^{j+1}\Rgf^{n+1} \right),
\]
\beq\label{D-expr}
\Dnn(t,x,y) =\exp\left( \sum_{n\ge 0} \sum_{j=0}^n \sum_{i\ge0}
  \frac 1 {n+1} {2n-j \choose n} {3n+i-j+1\choose 2n-j}  x^i   y^{j+1}\Rgf^{n+1}\right).
\eeq
The \gf\ of quartic Eulerian orientations, counted by vertices, is
\[
\Qgf(t)= \frac 1{3t^2}\left( t-3t^2-\Rgf(t)\right).
\]
\end{Theorem}
\begin{proof}
We take for $\Pnn$, $\Cnn$ and $\Dnn$ the above series, and define
$\Pgf$, $\Cgf$ and $\Dgf$ by~\eqref{natural-nn}. Since $\Rgf=O(t)$,
these three series are easily seen to belong respectively to the rings
$\qs[[y, t]]$, $\qs[x][[y,t]]$ and $\qs[[x,y,t]]$, as required by
Theorem~\ref{thm:systemQ}. Thus it suffices to check that the first four
equations of Theorem~\ref{thm:systemQ} hold, or, equivalently, that
  \begin{align}
\Pnn(t,y)&=\frac{1}{y}[x^1]\Cnn(t,x,y), \nonumber\\
\Dnn(t,x,y)&=\frac{1}{1-\Cnn\left(t,\frac{1}{1-x},y\right)}, \nonumber\\
\Dnn(t,x,y)&=1+y\, [x^{\geq0}]\left(\Dnn(t,x,y)\left(\frac{1}{x}\Pnn\left(t,\frac{1}{x}\right)+[y^{1}]\Dnn(t,x,y)\right)\right), \label{eq:D}
\\
[y^{1}]\Dnn(t,x,y)&=\frac{1}{1-x}\left(t+2[y^2]\Dnn(t,x,y)-([y^1]\Dnn(t,x,y))^2\right). \nonumber
\end{align}
Note that the first three equations do not involve {explicitly} the
variable $t$:
we will prove them without
resorting to the definition~\eqref{def:RQ} of $\Rgf$. 

The first equation is straightforward. For the second one, it suffices
to prove that for all $j \le n$,
\[
\sum_{i=0}^{2n-j} {3n-i-j\choose n} \frac 1{(1-x)^{i+1}}
= \sum_{i\ge 0} {3n+i-j+1\choose 2n-j} x^i.
\]
This follows by expanding the left-hand side in $x$ and
using the classical  identity, taken for $k=3n-j$:
\beq\label{not-chu}
\sum_{i=0}^{k-n}{k-i \choose n}{\ell+i\choose \ell}={k+\ell+1\choose n+\ell+1}={k+\ell+1\choose k-n}.
 \eeq

We now come to the third, and most interesting,
equation. Our first observation is that, in the expression~\eqref{D-expr} of $\Dnn(t,x,y)$,
the sum over $i$ is a rational function of $x$:
 \begin{align}
   \sum_{i\ge 0} {3n+i-j+1  \choose 2n-j} x^i 
& = \sum_{k\ge n+1}   {2n-j+k \choose 2n-j} x^{k-n-1} \nonumber
\\
&= \frac 1{x^{n+1} (1-x)^{2n-j+1} }- \sum_{\ell=0}^n {3n-\ell-j\choose
  2n-j} \frac1{x^{\ell+1}}.\label{id:ratQ}
 \end{align}
We note that the sum over $\ell$ {in the above expression} is a polynomial in
$1/x$, with no constant term. Let us denote it by
$L_{j,n}(1/x)$. The expression of
$\Dnn$ thus reads
\begin{align}
  \Dnn(t,x,y)&= \exp\left( \sum_{n\ge 0} \sum_{j=0}^{n} \frac 1 {n+1}
  {2n-j \choose n} y^{j+1} \Rgf^{n+1} \left( \frac 1{x^{n+1}
      (1-x)^{2n-j+1} }-L_{j,n}(1/x)\right)\right) \nonumber \\
&= \exp\big( A( \Ugf, z)-B(\Rgf,1/x,y)\big),\label{Dnn-expr}
\end{align}
where
\[
\Ugf= \frac{\Rgf}{x(1-x)^2}, \qquad z=(1-x)y,
\]
\beq\label{A:def}
 A( u, z)=
\sum_{n\ge 0} \sum_{j=0}^n \frac 1 {n+1} {2n-j\choose n} 
z^{j+1}u^{n+1}
\eeq
and
\[
B(r,1/x,y)=\sum_{n\ge 0} \sum_{j=0}^n \frac
  1 {n+1}{2n-j\choose n}L_{j,n}(1/x) y^{j+1} r^{n+1}.
\]
 By extracting  the coefficient of $y$ from~\eqref{Dnn-expr}, we find
\begin{align}
  [y^1]\Dnn(t,x,y)&= (1-x) \sum_{n\ge 0} \frac 1{n+1} {2n\choose n} \Ugf^{n+1}- \sum_{n\ge 0} \frac
  1 {n+1}{{2n\choose n}}L_{0,n}(1/x) \Rgf^{n+1},\nonumber
\\
&=(1-x) \,\Ugf \Cat(\Ugf) -\frac 1 x \, \Pnn\left(t,\frac 1 x\right),\label{id:PQ}
\end{align}
where $\Cat(u)$ is the Catalan series $\sum {2n \choose n
}\frac{u^n}{n+1}$ {and $\Pnn$ is given by~\eqref{P-expr}} (we have used the fact that ${2n\choose
  n}{3n-\ell\choose 2n}= {2n-\ell \choose n}{3n-\ell \choose n}$).
The identity~\eqref{eq:D} that we have to prove thus reads
\[
1=[x^{\ge 0}] \big( \Dnn(t,x,y) (1-z\Ugf \Cat(\Ugf))\big), 
\]
where we still denote $z=(1-x)y$. Equivalently, in view of~\eqref{Dnn-expr}:
\[
1= [x^{\ge 0}] \Big(
\exp(A(\Ugf,z))  \left(1-z\Ugf \Cat(\Ugf)\right)
\exp\left(- B\big(\Rgf, 1 /x ,y\big)\right)
\Big).
\]
We will prove below in Lemma~\ref{lem:Cat} that 
\[
\exp(A(\Ugf,z))  \left(1-z\Ugf \Cat(\Ugf)\right)=1,
\]
which, given that $B(\Rgf,1/x,y)$ only involves negative powers of $x$,
concludes the proof of the third identity.

Consider now the fourth equation of the system.  Given that
the second equation holds, what we need to prove can be rewritten as:
\[
[y^{1}]\Cnn(t,x,y)=x\left(t+2[y^2]\Cnn(t,x,y)+([y^1]\Cnn(t,x,y))^2\right).
\]
Let us write $\Cnn(t,x,y)=1-\exp(-T(t,x,y))$. Then the above identity
reads:
\[
[y^1]T(t,x,y) -2x [y^2]T(t,x,y)=tx.
\]
A direct calculation gives
\[
[y^1]T(t,x,y) -2x [y^2]T(t,x,y)=x \sum_{n\ge 0} \frac 1 {n+1}{2n \choose
  n}{3n \choose n} \Rgf^{n+1},
\]
which is precisely $xt$, by definition~\eqref{def:RQ} of the series $\Rgf$.

\medskip
We have thus proved the announced expressions of the series $\Pnn$,
$\Cnn$ and $\Dnn$, which in turn characterise the \gfs\ $\Pgf$,
$\Cgf$ and $\Dgf$ of patches of various types (see~\eqref{natural-nn}). We still have to
express the \gf \ $\Qgf(t)$ of quartic Eulerian orientations in terms
of $\Rgf(t)$.  The last equation of Theorem~\ref{thm:systemQ} now
reads
\begin{align*}
  \Qgf(t)&=\frac{1}{t^2}[y^1]\Pnn(t,y) -1
\\
&=\frac 1 {t^2} \sum_{n\ge 1} \frac 1 {n+1} {2n-1\choose
  n}{3n-1\choose n} \Rgf^{n+1} -1 
\\
&=\frac 1 {3 t^2} \sum_{n\ge 1} \frac 1 {n+1} {2n\choose
  n}{3n\choose n }\Rgf^{n+1} -1 
\\
&= \frac 1 {3 t^2}\left( t-\Rgf-3t^2\right)
\end{align*}
by definition of $\Rgf$.
\end{proof}
It remains to prove the following lemma, used in the above proof.
\begin{Lemma} \label{lem:Cat}
  For any indeterminates $u$ and $z$, the Catalan series
  $\Cat(u)=\sum_{n\ge 0}{2n\choose n} u^n/(n+1)$ and the
  series $A(u,z)$ defined by~\eqref{A:def} are related by:
\[
\exp(A(u,z))  \left(1-zu \Cat(u)\right)=1.
\]
\end{Lemma}
\begin{proof}
  Equivalently, what we want to prove reads
\[
A(u,z)= \log \frac1{1-zu \Cat(u)}=\sum_{j\ge 0} \frac {z^{j+1}}{j+1} (u\Cat(u))^{j+1}.
\]
Comparing with the expansion in $z$ of $A(u,z)$ (see~\eqref{A:def}),
what we want to show is
\[
[u^{n+1}] (u\Cat(u))^{j+1} = \frac{j+1}{n+1}{2n-j\choose n}.
\]
This follows from the Lagrange inversion formula~\cite[p.~732]{flajolet-sedgewick}, applied to
\[
F(u):=u\Cat(u)= \frac u{1-F (u)}.
\]
\end{proof}

\noindent
{\bf Remark.} The above lemma is {a special case of a
  general identity which} relates the enumeration of two classes of one-dimensional
  lattice paths, sharing the same step set, both constrained to end at a non-negative position. For the
  first class there is no other condition, while for the second class
  the path is not allowed to visit any negative point. The \gfs\ of
  these two classes, counted by the number of steps (variable $z$) and
  the final position (variable $u$) are respectively denoted by
  $W^+(z,u)$ and $F(z,u)$. Then, on p.~51
  of~\cite{banderier-flajolet}, {the following identity appears}
\[
F(z,u)=\exp\left( \int_0^z \left( W^+(t,u)-1\right)\frac{\text{d}t}t\right).
\]
When the only allowed steps are $+1$ and $-1$, this reads, using a
standard factori{z}ation on non-negative paths into \emm Dyck paths,
(counted by $\Cat(z^2)$):
\[
\frac{\Cat(z^2)}{1-zu\Cat(z^2)} =\exp\left(\sum_{n\ge 1}\sum_{j=0}^n
\frac1{2n-j} {2n-j \choose n} z^{2n-j} u^j\right).
\]
Upon dividing this identity by its specialization at $u=0$, we obtain
\[
\frac{1}{1-zu\Cat(z^2)} =\exp\left(\sum_{n\ge 1}\sum_{j=1}^n
\frac1{2n-j} {2n-j \choose n} z^{2n-j} u^j\right).
\]
Now some elementary transformations (involving replacing $u$ by $uz$,
then $z$ by $\sqrt z$, and finally swapping $u$ and $z$) shows that
this is equivalent to our lemma.

\section{A bijection}
\label{sec:bij}
In this short section, we first recall a bijection of Ambj\o rn and Budd~\cite{ambjorn-budd} that
sends labelled quadrangulations onto certain  maps carrying
integer labels on vertices (these maps are more general than the
labelled maps of Definition~\ref{def:labelled-map}). A specialization
of this bijection sends  \emm  certain, labelled
quadrangulations (those in which every face contains three labels)
onto labelled maps, which, as we have seen, are equinumerous with general
Eulerian orientations.
This is one of the key steps in the proof of Theorem~\ref{thm:gen}.
The  Ambj\o rn and Budd bijection, which generalizes the Cori-Vauquelin-Schaeffer
bijection between quadrangulations and certain labelled
trees~\cite{schaeffer-these,chassaing-schaeffer},  can also be seen to be equivalent to an earlier bijection
of Miermont~\cite{miermont-tesselations}. We refer to~\cite{bouttier-fusy-guitter} for a rich overview of
Schaeffer-like bijections.

The  Ambj\o rn and Budd bijection, which we denote by $\Phi$,  starts
from a labelled quadrangulation $Q$. The edges of $Q$ are  dashed in
our figures. The construction, illustrated on the left of 
Figure~\ref{fig:bijection}, takes place independently
in every face of $Q$, and in each face,  coincides with
Schaeffer's construction of labelled
trees~\cite{chassaing-schaeffer}: a new (solid) edge is created in
every face of $Q$, and its position depends on whether the face contains three of two distinct labels\footnote{If the outer face has three disctinct labels and  is drawn as the infinite face, the solid edge that we add still has an edge  from $\ell+1$ to $\ell+2$ on its right, now in clockwise order around the outer face.}. 
  A complete example  is shown on the right of   Figure~\ref{fig:bijection}.

\begin{figure}[htb]
  \centering
   \scalebox{0.9}{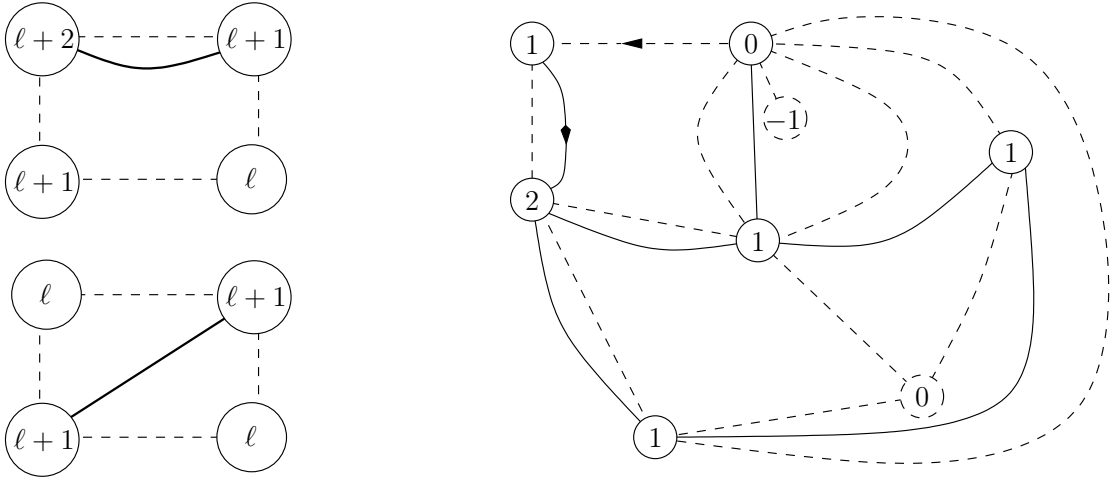}
  \caption{Left: Construction of $\Phi$ in a face of $Q$. Right: A
    labelled quadrangulation $Q$ (dashed
    edges) and the
    associated map $M=\Phi(Q)$ (solid edges)
    superimposed.  Two white vertices of $Q$, namely its
    local minima, shown in dashed disks, 
    disappear when constructing $M$.}
  \label{fig:bijection}
\end{figure}

Observe that each vertex of $Q$ that is not a local minimum is
joined to at least one other vertex.
Since the root edge of $Q$ is oriented from $0$ to $1$,  the co-root vertex
must be  joined to another vertex $v$ by an edge 
 located in the co-root face of $Q$. We orient
this edge towards~$v$: this will be 
the root edge of the new object. Finally, we delete all edges of $Q$,
and also all vertices of $Q$ 
that have become isolated: they are those whose label is a local
minimum. We denote by $\Phi(Q)$ the resulting object, which is a planar
graph embedded in the plane, with a root edge that starts from a   vertex
labelled 1.

\begin{Proposition}[Thm.~1 in~\cite{ambjorn-budd}] \label{prop:bij}
  The transformation $\Phi$ {bijectively sends} labelled quadrangulations to
  planar maps carrying
  integer labels on vertices, differing by $0, \pm1$ along edges,
  having root vertex labelled $1$. Moreover, if $\Phi(Q)=M$, then the
  number of edges and faces in $M$ are given by
\[
\ee(M)=\ff(Q), \qquad \ff(M)=\vv_{\min}(Q),
\]
where $\vv_{\min}(Q)$ denotes the number of local minima in $Q$. The first identity can be refined as follows: a face
  of $Q$ in which only two different labels occur gives rise to an edge of $M$
  with increment $0$, while a face where three different labels occur gives
  rise to an edge with increment~$\pm1$. 
\end{Proposition}

\begin{figure}[b]
  \centering
   \scalebox{0.7}{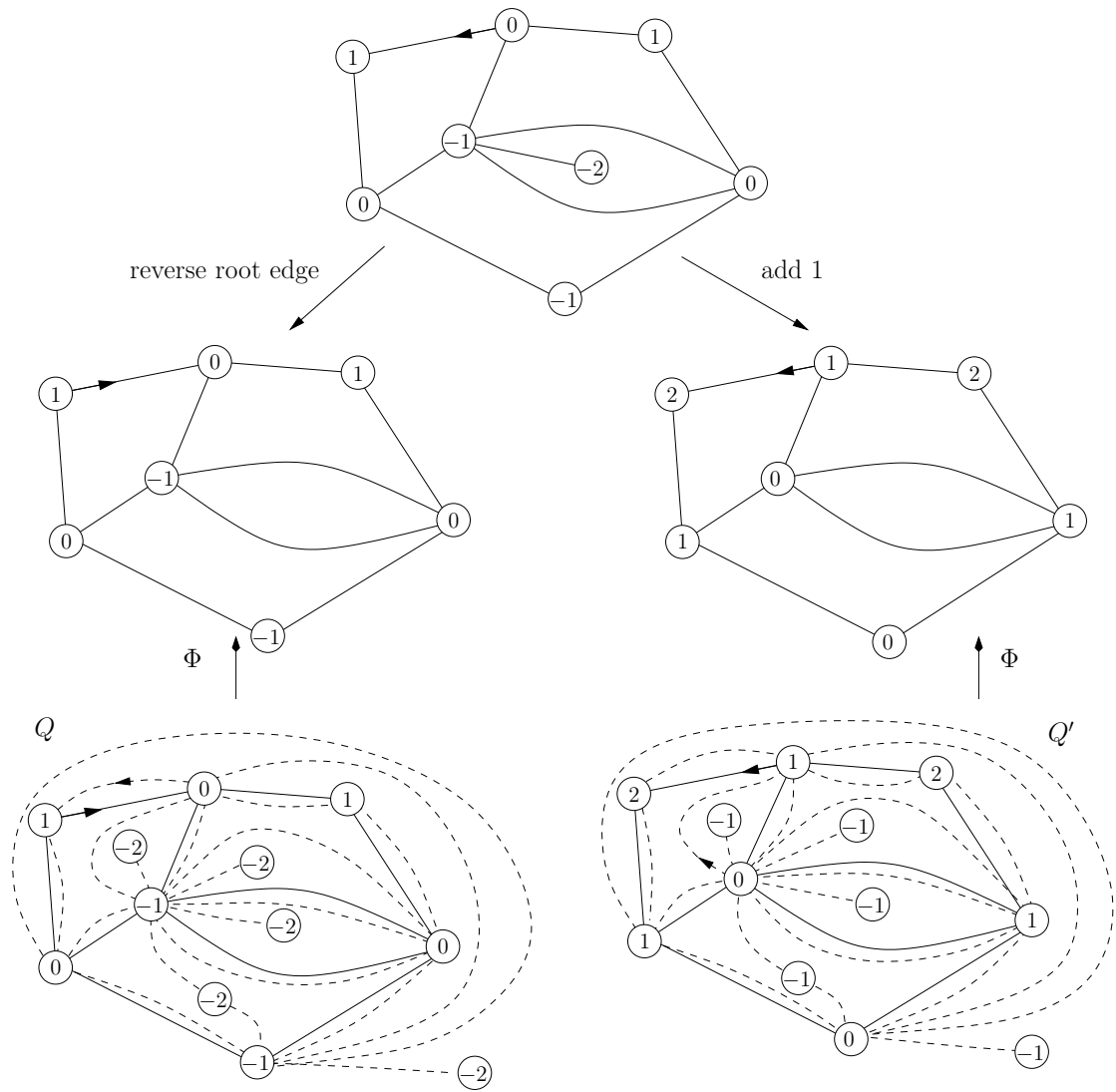}
  \caption{One labelled map gives two maps with root vertex 1, which
    are the images by $\Phi$ of two colourful labelled
    quadrangulations $Q$ and $Q'$ (in dashed edges). They only differ by a shift of labels and a
    change in the root edge.}
  \label{fig:1to2}
\end{figure}

Of particular importance will be labelled quadrangulations in which every face
  (including the outer one) contains
  three distinct labels: we call them \emm colourful,.  Take  a colourful
labelled quadrangulation $Q$. By the above proposition, the map
$M:=\Phi(Q)$ has all increments equal to~$\pm1$. 
Its root vertex is labelled $1$, hence the root edge is labelled either from $1$ to
$0$, or from~$1$ to~$2$. In the former case, reversing the direction
of the root edge gives a labelled map (in the sense of
Definition~\ref{def:labelled-map}). In the latter case, subtracting
$1$ from every label gives a labelled map. 
Conversely,  take a labelled map $L$, and reverse the orientation of its root edge: this gives
 a map of the form $\Phi(Q)$, in which the root edge has
labels~1 and 0 (Figure~\ref{fig:1to2}, left). Alternatively, one can
add $1$ to every label of $L$: the
resulting map  is of the form $\Phi(Q')$, and its root edge {has} labels 1 and 2
(Figure~\ref{fig:1to2}, right). 
 This
gives a 2-to-1 correspondence between colourful labelled
quadrangulations and labelled maps. This  will
be the key in our enumeration of general Eulerian orientations.

\begin{Corollary}\label{cor:colourful}
  The number of  colourful labelled quadrangulations with $n$ faces and
  $k$ local minima equals twice the number of labelled maps with $n$
  edges and $k$ faces, or equivalently, twice the number of Eulerian
  orientations with $n$ edges and $k$ vertices.
\end{Corollary}

If we
 start from a general labelled quadrangulation $Q$, possibly
   containing  faces with  only two labels, then we can still
 apply the duality rule of Figure~\ref{fig:duality} to the map $\Phi(Q)$ (carrying labels), with
 the additional rule that we do not orient an edge that lies between
 two faces with the same label. In this way we obtain an Eulerian \emm partial,
 orientation, that is, a map in which \emm some, edges are
 oriented, in such a way that there are as many incoming as outgoing edges
 at any vertex (Figure~\ref{fig:4-V-partial}).

\begin{Corollary}
  There is a bijection between quartic Eulerian orientations  with
  $n$ vertices ({in which the root edge is}  oriented canonically) and Eulerian  partial orientations  with $n$ edges
  (with no orientation requirement on the root edge).
\end{Corollary}

\begin{figure}[ht]
  \centering
     \scalebox{0.7}{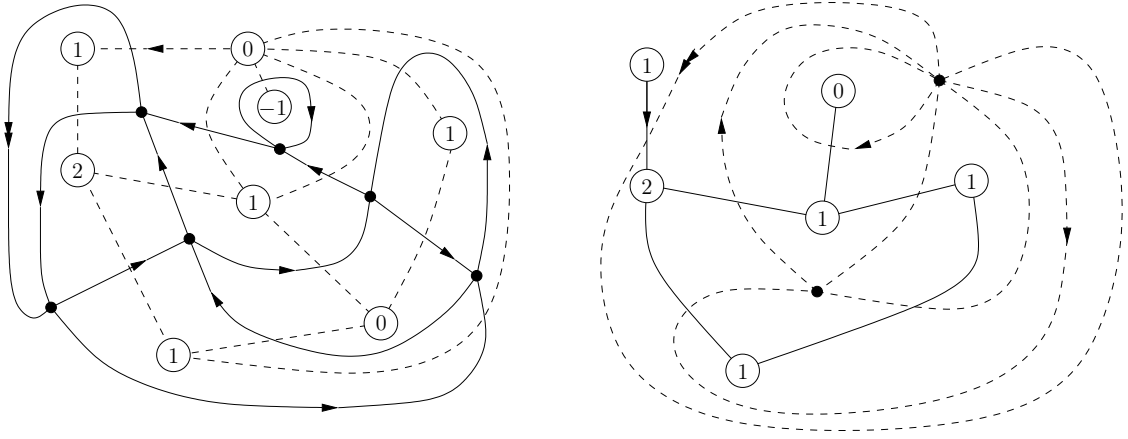}
  \caption{From quartic Eulerian orientations to Eulerian \emm
      partial, orientations. Left: a quartic Eulerian orientation, shown by solid edges,
    and the dual labelled quadrangulation $Q$ (dashed edges). This is
    the quadrangulation of Figure~\ref{fig:bijection}, which also shows the map
    $M=\Phi(Q)$.  Right:
    upon  re-applying duality to $M$ (shown in solid lines), one obtains an
    Eulerian partial orientation of its dual    (dashed edges).}
  \label{fig:4-V-partial}
\end{figure}

\section{Functional equations for general Eulerian  orientations}
\label{sec:func-gen}

In this section we will characterise the generating function $\Qc(t)$
of colourful labelled quadrangulations (which, by
  Corollary~\ref{cor:colourful}, is twice the \gf\ of Eulerian
orientations) by a system of functional equations.  As one might expect,
we adapt  the system of
Theorem~\ref{thm:systemQ} to the colourful setting. However,  the third equation and the initial conditions are simpler in the colourful case.

\begin{Theorem}\label{thm:systemG}
  There exists a unique $3$-tuple of series, denoted   $\Pgf(t,y)$, $\Cgf(t,x,y)$
and $\Dgf(t,x,y)$, belonging respectively to $\qs[[y,t]]$,
  $\qs[x][[y,t]]$ and $\qs[[x,y,t]]$, and satisfying the
  following equations:
\begin{align*}
\Pgf(t,y)&=\frac{1}{y}[x^1]\Cgf(t,x,y),\\
\Dgf(t,x,y)&=\frac{1}{1-\Cgf\left(t,\frac{1}{1-x},y\right)},\\
\Cgf(t,x,y)&=xy[x^{\geq0}]\left(\Pgf(t,tx)\Dgf\left(t,\frac{1}{x},y\right)\right),
\end{align*}
together with the initial condition $\Pgf(t,0)=1$.

The \gf\ that counts colourful labelled quadrangulations by faces
is
\[
\Qc(t)=[y^1]\Pgf(t,y)-1.
\]
By  Corollary~\ref{cor:colourful}, $\Qc(t) =2 \Ggf(t)$, where $\Ggf(t)$   counts Eulerian
orientations  by edges.
\end{Theorem}
 
\medskip
\noindent{\bf Remark.} As with the system of
Theorem~\ref{thm:systemQ},  the conditions  on the series
$\Pgf$, $\Cgf$ and $\Dgf$ make the operations that occur in the above
equations  well defined. The extraction of the coefficient of
  $x^1$, and the replacement of $x$ by $1/(1-x)$, are justified as for
   the previous system. In  the third equation, the term $\Pgf(t,tx)\Dgf\left(t,\frac{1}{x},y\right)$
must be seen as a power series in $t$ and $y$ whose coefficients are
Laurent series in $1/x$. The extraction of the non-negative part in
$x$ then yields an element of $\qs[x][[y,t]]$.

\medskip
As before, the series  $\Pgf$, $\Cgf$ and $\Dgf$ of Theorem~\ref{thm:systemG}
count certain labelled maps. Recall the definition of patches,
C-patches and D-patches (Definition~\ref{def:patches}). Generalizing
the definition of colourful quadrangulations introduced in Section~\ref{sec:bij}, we say that 
a patch  (or a D-patch) is  {\em colourful} if each inner
quadrangle contains $3$ distinct labels.
We define $\Pgf(t,y)$, $\Cgf(t,x,y)$ and $\Dgf(t,x,y)$ to be
respectively the \gfs\ of colourful patches, colourful C-patches and colourful D-patches, where $t$
counts inner quadrangles, $y$ the outer degree (halved), and $x$
either the degree of the root vertex (for  C-patches)  or the number of
inner digons (for  D-patches). The equations
\begin{align*}
\Pgf(t,y)&=\frac{1}{y}[x^1]\Cgf(t,x,y)\\
\Dgf(t,x,y)&=\frac{1}{1-\Cgf\left(t,\frac{1}{1-x},y\right)},\\
\Qc(t)&=[y^1]\Pgf(t,y)-1,
\end{align*}
have identical proofs to those in Section~\ref{sec:func4}, except that
the patches, D-patches and quadrangulations in the proofs are
restricted to being colourful (see Lemmas~\ref{PfromC},
  \ref{DfromC} and~\ref{lem:Qgf}). 

\medskip
\noindent{\bf Remark.} The third equation of Theorem~\ref{thm:systemQ},
\beq\label{eq:D-heavy}
\Dgf(t,x,y)=1+y \, [x^{\geq0}]\left(\Dgf(t,x,y)\left(\frac{1}{x}\Pgf\left(t,\frac{t}{x}\right)+[y^{1}]\Dgf(t,x,y)\right)\right)
\eeq
also holds in the colourful setting, with the same proof as
before (because in the proof of  Lemma~\ref{Dequation}, all
  quadrangles that come from digons are automatically colourful).   Its
natural complement, which is the fourth  equation {of} Theorem \ref{thm:systemQ} (the
initial condition) has no clear colourful counterpart: indeed, the
relabelling of vertices  that we use in Lemma~\ref{Dout2}, and more
precisely in the fourth case
of Figure~\ref{fig:Dout2_cases}, transforms the colourful quadrangles
incident to the root vertex into bicoloured quadrangles (and vice
versa).  We could  instead use the initial condition
$[y^1]\Cgf(t,x,y)=x \Pgf(t,tx)$, which can be proved by taking
  a colourful C-patch
of outer degree~2 and deleting the
root vertex and all  incident edges, {then decreasing each label by $1$ (Figure~\ref{fig:Cresinitial})}. However, the third equation of Theorem~\ref{thm:systemG} is simpler
than~\eqref{eq:D-heavy}, and  also relies on a simpler
construction. 
\begin{figure}[ht]
\setlength{\captionindent}{0pt}
 \includegraphics[scale=0.7]{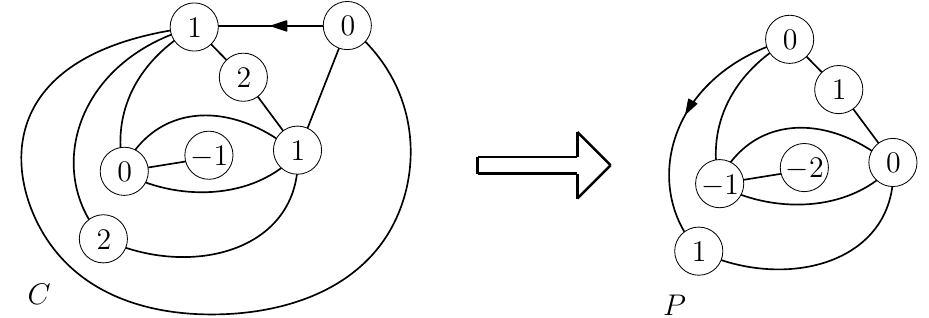} 
   \caption{A colourful C-patch $C$ with outer degree $2$ and the corresponding patch $P$.}
   \label{fig:Cresinitial}
\end{figure}

\medskip

In order to prove the third equation of Theorem~\ref{thm:systemG}, we need some
analogues of minus-patches from Section \ref{sec:func4}, which we call \emm
shifted patches,.

\begin{Definition}
  A {\em shifted patch} is a
  map obtained from  a patch by replacing each label $\ell$ with $\ell+1$.
\end{Definition}
We now describe a way to extract a shifted subpatch from a patch,
which  parallels the extraction of a minus-patch of
Definition~\ref{def:minus-sub}. One minor difference is that we do not
need conditions on the neighbours of the root vertex, so that we
define shifted subpatches for any patch (although we will only 
extract them from C-patches later).

\begin{Definition}
 Let $P$ be a patch and let $c$ be an outer corner of $P$ at a vertex
 $v$ labelled $1$. We define the {\em shifted subpatch of $P$ rooted
 at $c$} as follows. First, let $S'$ be the maximal connected submap
 of $P$ that contains $v$ and consists of vertices   labelled
   $1$ or more. Let $S$ be the submap of $P$ that contains $S'$ and all
 edges and vertices within its boundary (assuming the root face is
 drawn as the infinite face). The map $S$, which we root at the corner
 inherited from $c$, is the {\em shifted subpatch of $P$ rooted at
 $c$}. 
\end{Definition}
An example is shown on the left of
  Figure~\ref{fig:CinPD_proof}, where the shifted subpatch $S$ is drawn with
  thick lines. It is easily shown  that  $S$ is, as it should {be}, a shifted patch. The argument is the
same as for minus-subpatches (in that case, the condition on the
neighbours of the root having labels $1$ was there to prevent the
minus-subpatch to absorb the root vertex; this cannot happen with the
shifted subpatch, whose boundary  only contains positive labels). Every edge in $P$ {that}
connects a vertex in  $S$ to a vertex not in $S$ must have endpoints
labelled $1$ (in $S$) and $0$ (not in~$S$), and conversely every
vertex labelled 1 on the boundary of $S$ is joined to a vertex
labelled $0$ out of $S$. We can  contract $S$ it into a
single vertex $v_{1}$  labelled $1$. This vertex is only adjacent to
vertices labelled $0$ in the resulting map $L$, and the number of
digons incident to $v_{1}$ is half the outer degree of $S$. The
  outer degrees of $L$ and $P$ coincide, because no edge of the
  boundary of~$P$ has been contracted.

As in the case of minus-subpatches, we can uniquely reconstruct the patch $P$
and its marked corner $c$ if we are given the shifted patch $S$ and
the contracted map $L$, together with its outer corner inherited from
$c$. The idea is again to attach the edges incident to $v_1$ in $L$
around the shifted patch $S$, as illustrated (in the case of minus-patches) in Figure~\ref{fig:decontract}.

We are now ready to prove the third equation of Theorem~\ref{thm:systemG}.

\begin{figure}[ht]
\setlength{\captionindent}{0pt}
   \includegraphics[scale=0.7]{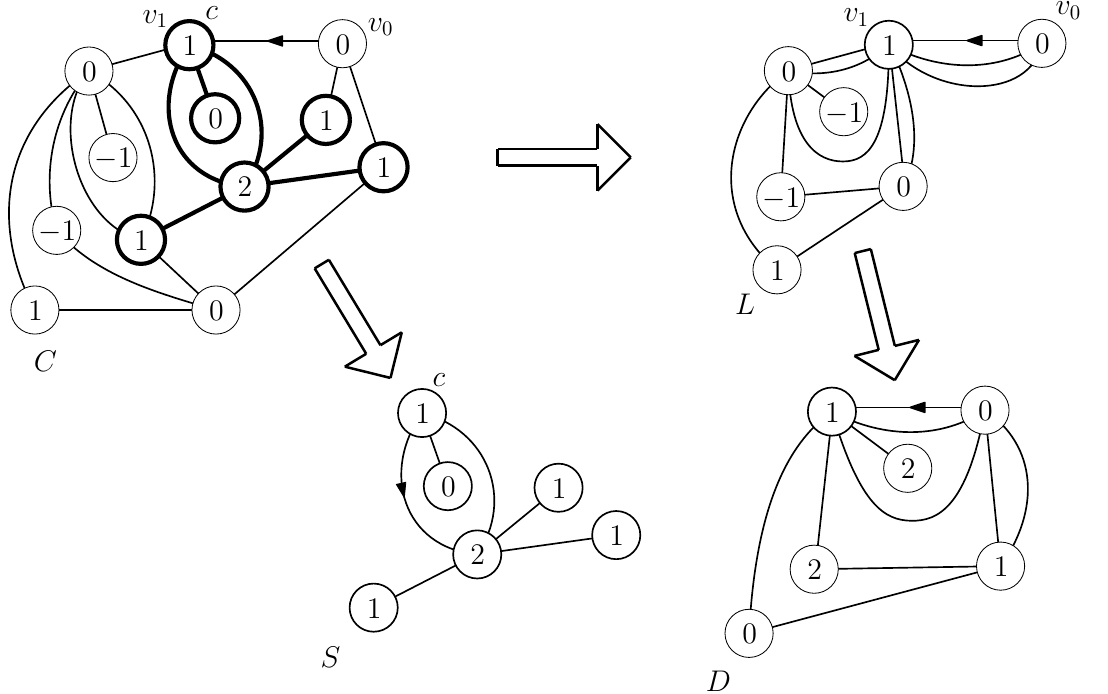} 
   \caption{On the left is a C-patch ${C}$. The
     other  maps shown are the shifted subpatch $S$, the contracted
     map $L$ and the D-patch $D$ obtained by deleting $v_0$, as described in the proof
     of Lemma~\ref{lem:CG}.}
   \label{fig:CinPD_proof}
\end{figure}
\begin{Lemma} \label{lem:CG}
The generating functions $\Pgf$, $\Cgf$ and $\Dgf$ satisfy the equation
\[
\Cgf(t,x,y)=xy[x^{\geq0}]\left(\Pgf(t,tx)\Dgf\left(t,\frac{1}{x},y\right)\right).\]
\end{Lemma}
\begin{proof}
Let $C$ be any colourful C-patch. Let $v_{0}$ and $v_{1}$ be the root
vertex and co-root vertex of~$C$, and let $c$ be the outer corner of
$v_{1}$ {that} is immediately anticlockwise of the root edge (we refer
to Figure~\ref{fig:CinPD_proof} for an illustration). Let $S$
be the shifted subpatch of ${C}$ rooted at $c$ and let~$L$ be the
labelled map obtained from $C$ by contracting the subpatch $S$ to a single vertex, still
denoted $v_{1}$. Then in $L$, the root vertex $v_{0}$ is only adjacent
to vertices labelled 1 (because this was already true in $C$),
and the co-root vertex {$v_1$} is only adjacent to vertices labelled 0
(because of the contraction). 

Recall that all inner faces of $L$ are either digons or
  quadrangles. We want to prove that in~$L$, {the root vertex} $v_0$ is not incident to any inner
  quadrangle. Assume that such a quadrangle exists. Since $v_{0}$ only
  shares one corner with the outer face of $L$ (this was the case for
  $C$ already), {and since $v_0$ is adjacent to $v_1$}, one such quadrangle must be incident to
  $v_1$. But the above label conditions on the neighbours of $v_0$ and
  $v_1$ force this quadrangle to have labels 0 and 1 only, in $L$  and
  thus in $C$. This  contradicts the fact that $C$ is colourful.
Hence $v_0$ is only adjacent to {inner} digons (and to the outer face), and since it shares only one
corner with the outer face, its only neighbour {in $L$} is  $v_{1}$ (see Figure~\ref{fig:CinPD_proof}).

Let $D$ be the labelled map constructed from $L$ by moving the root
edge anticlockwise one place around the outer face, removing the old
root vertex $v_{0}$ of $L$ and all incident edges, and replacing each
vertex label $\ell$ with $1-\ell$. Now the root vertex has label $0$,
and all its neighbours have label $1$. The outer face {still has} labels
0 and 1. Each inner quadrangle of $D$
corresponds to an inner quadrangle of ${C}$, and is therefore
colourful. Hence $D$ is a colourful D-patch. 

Let $2j$ be the outer degree of $S$. Then $j$ is also
  the number of inner digons in $L$.  Let $i\le j$ be the number of
{inner} digons left in $D$ after deleting $v_0$. Then the number
of {inner} digons in $L$ {that} are incident to $v_{0}$ (and $v_1$) is $j-i$. Therefore,
the degree of $v_{0}$ in both $L$ and ${C}$ is $j-i+1$. 

{Conversely, taking  a colourful D-patch $D$ with $i$ inner digons and a colourful shifted
patch $S$ of outer degree $2j$ with $j\ge i$ we construct the corrresponding C-patch $C$ as follows:
\begin{itemize}
\item we first construct a map $D'$ by replacing each label $\ell$ in $D$ with $1-\ell$,
\item next we construct a map $L$ with a new vertex $v_0$, joined
  to the root vertex of {$D'$} by $j-i+1$ edges,
  \item finally we insert $S$ into
  $L$ to create the corresponding patch ${C}$ (the
choice of the corner where the subpatch extraction takes place being
canonical). 
\end{itemize}}
As already explained, the degree of the root vertex in $C$ is
$j-i+1$. The outer degree of $C$ is the outer degree of $D$ plus 2,
and the number of inner quadrangles in $C$ is $j$ plus the number of
inner quadrangles in $S$ and $D$. Hence, with the obvious notation,
\[
C(t,x,y)= \sum_{\od(S)\ge \dig(D)} x^{\od(S)-\dig(D)+1}y^{1+\od(D)}
t^{\qu(S)+\qu(D)+\od(S)},
\]
and this  gives the 
equation of the lemma, since shifted patches are counted by $\Pgf(t,y)$.
\end{proof}

\begin{proof}[Proof of Theorem~\ref{thm:systemG}.] Given that the
  initial condition $\Pgf(t,0)=1$ is obvious (it accounts for the
  atomic patch), we have now proved
  all functional equations. It remains to prove that, together with
  the conditions on the rings that contain $\Pgf$, $\Cgf$ and $\Dgf$, they determine a unique 3-tuple of series.  Let us denote by $p_{j,n}$ the coefficient of $y^j t^n$ in
$\Pgf(t,y)$, and similarly for $\Cgf$ and $\Dgf$. These quantities
should be thought {of} as elements of $\qs$ (for $\Pgf$),
of $\qs[x]$ (for $\Cgf$) and of $\qs[[x]]$ (for $\Dgf$).  We will prove by
induction {on} $N\ge 0$ that
 $p_{j,n}$, $c_{j,n}$ and $d_{j,n}$ are completely determined for $j+n
 \le N$ --- we say \emm up to order, $N$.

First take  $N=0$. The third equation
of the system shows that $\Cgf$ is a multiple of $y$. In particular,
$c_{0,0}=0$. The second
equation then implies that ${\Dgf}-1$ is also a multiple of $y$. In
particular, $d_{0,0}=1$.
Finally, the initial condition ${\Pgf}(t,0)=1$ gives $p_{0,0}=1$.
Now assume that the induction hypothesis holds for some $N\ge 0$, and let us prove it for
$N+1$. 

The third equation, with its factor $y$, allows us to determine
${\Cgf}(t,x,y)$ up to order $N+1$. By construction, the coefficients $c_{j,N+1-j}$
 will be polynomials in $x$. Then the second equation gives
$\Dgf(t,x,y)$ up to the same order. The first equation seems to raise
a problem, because of the division by $y$. But, combined with the
third equation, it reads
\begin{align*}
\Pgf(t,y)&=[x^0] \Pgf(t,tx) \Dgf(t,1/x,y)\\
&=[x^0] \Pgf(t,tx)+ [x^0] \Pgf(t,tx) (\Dgf(t,1/x,y)-1)\\
&=\Pgf(t,0)+ [x^0] \Pgf(t,tx) (\Dgf(t,1/x,y)-1).
\end{align*}
Now $\Pgf(t,0)=1$ is known, and since $\Dgf(t,1/x,y)-1$ is a multiple
of $y$, knowing $\Pgf(t,tx)$ up to order $N$ and
  $\Dgf(t,1/x,y)$ up to order $N+1$ suffices to determine
$\Pgf(t,y)$ up to order $N+1$. This completes our
induction.
\end{proof}

\section{Solution for general Eulerian orientations}
\label{sec:sol-gen}
We are now about to solve the system of Theorem~\ref{thm:systemG},
thus proving, in particular, that the \gf\ $\Ggf(t)$ of 
Eulerian orientations is indeed given by Theorem~\ref{thm:gen}. As in
Section~\ref{sec:sol4}, the
third equation of the system leads us to introduce variants of the
series $\Pgf$, $\Cgf$ and $\Dgf$, defined again by
\beq\label{natural-nn-gen}
\Pnn(t,y)=t\,\Pgf(t,ty), \qquad \Cnn(t,x,y)=\Cgf(t,x,ty), \qquad
\Dnn(t,x,y)=\Dgf(t,x,ty).
\eeq
Of course, if we determine $\Pnn, \Cnn$ and $\Dnn$, then $\Pgf, \Cgf$
and $\Dgf$ are completely determined as well.
\begin{Theorem}\label{thm:solg}
Let $\Rgf(t)\equiv \Rgf$ be the unique formal power series with constant
  term $0$ satisfying
\[
t= \sum_{n \ge 0} \frac 1 {n+1} {2n \choose n}^2 \Rgf^{n+1}.
\]
  Then the above series $\Pnn$, $\Cnn$ and $\Dnn$ are:
\[
\Pnn(t,y)=\sum_{n\ge 0}\sum_{j=0}^n \frac 1{n+1} {2n\choose n}{2n-j
  \choose n} y^j \Rgf^{n+1},
\]
\[
\Cnn(t,x,y) =1 -\exp\left( 
- \sum_{n\ge 0} \sum_{j=0}^n \sum_{i=0}^{n}
  \frac 1 {n+1} {2n-i \choose n} {2n-{j}\choose n}  x^{i+1}  y^{j+1}\Rgf^{n+1} \right),
\]
\[
\Dnn(t,x,y) =\exp\left( \sum_{n\ge 0} \sum_{j=0}^n \sum_{i\ge0}
  \frac 1 {n+1} {2n-j \choose n} {2n+i+1\choose n}  x^i   y^{j+1}\Rgf^{n+1}\right).
\]
The \gf\ of  Eulerian orientations, counted by edges, is
\[
{\Ggf}(t)= \frac 1{4t^2}\left( t-2t^2-\Rgf(t)\right).
\]
\end{Theorem}
\noindent{\bf Remark.} Observe that the series $\Cnn(t,x,y)$ is
  symmetric in $x$ and $y$. Let us give a combinatorial explanation
  for this, illustrated in Figure \ref{fig:CpfromC_example}. Consider
  a colourful C-patch $C$ with root vertex $v_{0}$, and form a
  colourful labelled quadrangulation $C'$ by adding a vertex $v_{2}$
  with label $2$ to the outer face of $C$ and joining it to each outer
  corner of $C$ labelled 1.  The generating function
  $\Cnn(t,x,y)=\Cgf(t,x,ty)$ then counts the possible objects $C'$ by
  the number of quadrangles (variable~$t$), the degree of the root
  vertex $v_{0}$ (variable $x$) and the degree
  of the new vertex $v_{2}$ (variable~$y$). Moreover, the object $C'$ can be any
  colourful quadrangulation in which the outer face has labels $0$,
  $1$, $2$, $1$ and each vertex that neighbours either $v_{0}$ or
  $v_{2}$ is labelled $1$.  The transformation $\ell \mapsto 2-\ell$
  then explains why the   generating function $\Cnn(t,x,y)$ is
  symmetric in $x$ and   $y$.
  \begin{figure}[ht]
\setlength{\captionindent}{0pt}
   \includegraphics[scale=0.7]{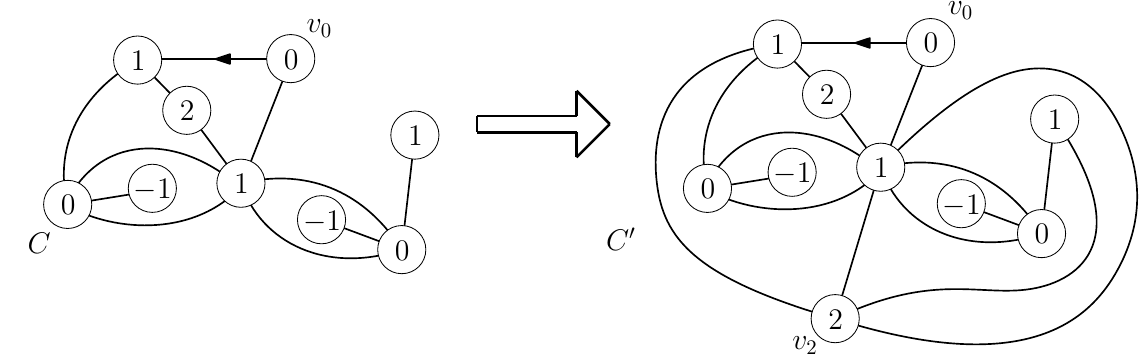} 
   \caption{An example of the transformation from a colourful C-patch $C$ to a colourful labelled quadrangulation $C'$ from the remark below theorem \ref{thm:solg}.
}
   \label{fig:CpfromC_example}
\end{figure}
\begin{proof}[Proof of Theorem~\ref{thm:solg}]
We argue as in the proof of Theorem~\ref{thm:solQ}. Defining $\Pnn$, $\Cnn$
and $\Dnn$ as above, we first observe that the series $\Pgf$, $\Cgf$
and $\Dgf$ defined by~\eqref{natural-nn-gen} belong respectively to the rings
$\qs[[y,t]]$, $\qs[x][[y,t]]$ and $\qs[[x,y,t]]$, as prescribed in
Theorem~\ref{thm:systemG}. Hence it suffices to prove that they
satisfy the desired system, which reads
  \begin{align}
\Pnn(t,y)&=\frac{1}{y}[x^1]\Cnn(t,x,y), \nonumber\\
\Dnn(t,x,y)&=\frac{1}{1-\Cnn\left(t,\frac{1}{1-x},y\right)}, \nonumber\\
\Cnn(t,x,y)&=xy [x^{\ge 0}] \Pnn(t,x) \Dnn(t, 1/x, y), \nonumber\\
\Pnn(t,0)&=t.\nonumber
\end{align}
Note that the first three equations do not  explicitly involve the
variable $t$: we will prove them without
resorting to the definition of $\Rgf$. But the fourth
equation, namely the initial condition $\Pnn(t,0)=t$, does
  involve $t$, and in fact holds precisely by definition of $\Rgf$.

The first equation is again straightforward, and the second follows
from~\eqref{not-chu} again. Now consider the third one. Since
it is more natural to handle series in $x$ rather than in $1/x$, we will
show instead that
\begin{align}
  \Cnn(t,1/x,y)&=y/x [x^{\le 0}] \Pnn(t,1/x) \Dnn(t, x, y)  \nonumber \\
&=  [x^{<0}] \left(y/x\,\Pnn(t,1/x) \Dnn(t, x, y)\right).\label{CPD-gen}
\end{align}
In order to prove~\eqref{CPD-gen}, we use
identities that are similar to those used in the proof
of~\eqref{eq:D}. The counterpart of~\eqref{id:ratQ} is:
\[
\sum_{i\ge 0} {2n+i+1\choose n} x^i =\frac 1{x^{n+1}(1-x)^{n+1}} -
\sum_{\ell=0} ^n {2n-\ell \choose n} \frac 1 {x^{\ell+1}}.
\]
 The counterpart of~\eqref{Dnn-expr} is: 
\beq\label{id:CD-G}
\Dnn(t,x,y) =\exp\left( A(\Ugf,y)\right) \big( 1-\Cnn(t, 1/x,y)\big) ,
\eeq
with $\Ugf=\frac{\Rgf}{x(1-x)}$
and $A(u,y)$ is still given by~\eqref{A:def}. This is indeed an analogue of~\eqref{Dnn-expr}, since
\[ 
1-\Cnn(t, 1/x,y)= \exp\left( 
- \sum_{n\ge 0} \sum_{j=0}^n \sum_{i=0}^{n}
  \frac 1 {n+1} {2n-i \choose n} {2n-{j}\choose n}  \frac 1{x^{i+1}}
  y^{j+1}\Rgf^{n+1} \right)
\]
can be written as $\exp(-B(\Rgf,1/x,y))$ where $B(\Rgf,1/x,y)$ only
involves negative powers of $x$.
By extracting  the coefficient of $y$ from~\eqref{id:CD-G}, we find the counterpart of~\eqref{id:PQ}: 
\beq\label{id:P-G}
[y] \Dnn(t,x,y) =  \Ugf \Cat(\Ugf)-\frac 1 x  \Pnn\left(t, \frac 1
  x\right),
\eeq
where $ \Cat(u)$ is still the Catalan series $ \sum_{n\ge 0}
\frac{u^n}{n+1}{2n \choose n}$.

With these identities at hand, we can now prove~\eqref{CPD-gen}:
\begin{align*}
  [x^{<0}] \left(y/x\,\Pnn(t,1/x) \Dnn(t,x,y) \right)
&=  [x^{<0}]\left(y\Dnn(t,x,y) \left(  \Ugf \Cat(\Ugf)-[y]
  \Dnn(t,x,y) \right)\right) \qquad \hbox{by } \eqref{id:P-G},
\\
&=  [x^{<0}] \left(y\Dnn(t,x,y) \  \Ugf
  \Cat(\Ugf)\right)
\\
&=  [x^{<0}] \left(-\Dnn(t,x,y)(1-y   \Ugf
  \Cat(\Ugf))\right)
\\
&=  [x^{<0}] \left(-\Dnn(t,x,y)\exp(-A(\Ugf,y)) \right)
  \qquad \qquad\hbox{by Lemma } \ref{lem:Cat},
\\
&= [x^{<0}] \left(- 1+ \Cnn(t,1/x,y)\right)  \qquad\qquad\qquad\qquad
  \hbox{by }\eqref{id:CD-G},
\\&=\Cnn(t,1/x,y).
\end{align*}

We have thus proved the announced expressions of $\Pnn$, $\Cnn$
  and $\Dnn$, which in turn characterise the \gfs\ $\Pgf$, $\Cgf$ and
  $\Dgf$ of colourful patches of various types. Now the last equation
  of Theorem~\ref{thm:systemG} gives
  \begin{align*}
    2\Ggf(t)= \Qc(t)&= \frac 1 {t^2} [y^1] \Pnn(t,y)-1
\\
&=\frac 1 {t^2} \sum_{n\ge 1}  \frac 1 {n+1}{2n\choose n}{2n-1\choose
  n} \Rgf^{n+1}-1
\\
&= \frac 1 {2t^2} \sum_{n\ge 1}  \frac 1 {n+1}{2n\choose
  n}^2\Rgf^{n+1}-1
\\
&= \frac 1 {2t^2} \left( t-\Rgf-2t^2\right).
  \end{align*}
The expression given in Theorem~\ref{thm:solg} (and in Theorem~\ref{thm:gen}) for
$\Ggf(t)$ follows.
\end{proof}

\section{Nature of the series and asymptotics}
\label{sec:asympt}

\subsection{Nature of the series}

We begin by proving that the series $\Qgf(t)$ and $\Ggf(t)$ that count
respectively 
quartic and general Eulerian orientations satisfy
non-linear differential equations of order 2, as claimed in
Theorems~\ref{thm:4} and~\ref{thm:gen}.
Both series are expressed in terms of a series $\Rgf$ that satisfies
\[
\Omega(\Rgf)=t,
\]
for some hypergeometric series $\Omega$. In the quartic case (Theorem~\ref{thm:4}), 
\beq\label{omega4}
\Omega(r)= \sum_{n\ge 0} \frac 1 {n+1}{2n\choose n}{3n\choose n} r^{n+1}
\eeq
satisfies
\[
6 \Omega(r)+r(27r-1)\Omega''(r) {=0},
\]
from which we derive that 
\[
{ \Rgf(27\,\Rgf -1) \Rgf ''}=6t\,\Rgf '^3.
\]
Using $3t^2\Qgf(t)=t-3t^2-\Rgf(t)$, this gives indeed a second order DE
for $\Qgf(t)$, of degree $3$.

For general Eulerian orientations (Theorem~\ref{thm:gen}), we still
have $\Omega(\Rgf)=t$, with
\beq\label{omegaG}
\Omega(r)= \sum_{n\ge 0} \frac 1 {n+1}{2n\choose n}^2 r^{n+1}
\eeq
satisfies
\[
4 \Omega(r)+r(16r-1)\Omega''(r){=0},
\]
from which we derive that 
\[
{ \Rgf (16\, \Rgf -1) \Rgf ''}=4t\, \Rgf '^3.
\]
Using $4t^2\Ggf(t)=t-2t^2-\Rgf(t)$, this gives a second order DE
for $\Ggf(t)$, of degree $3$.

The fact that neither $\Qgf(t)$ nor $\Ggf(t)$ solve a non-trivial \emm
linear, DE  will follow from the asymptotic behaviour of their
coefficients, established in the next subsection: indeed, the
logarithm occurring at the denominator prevents this behaviour from
being that of the coefficients of a D-finite series~\cite[p.~520
  and~582]{flajolet-sedgewick}.

We can also describe the nature of the multivariate series counting
patches. 
\begin{Proposition}
  The \gfs\ $\Pgf(t,y)$, $\Cgf(t,x,y)$ and $\Dgf(t,x,y)$ counting
  patches of various types, and expressed in Theorem~\ref{thm:solQ}
  through the identities~\eqref{natural-nn}, are
  D-algebraic. The same holds for their
  colourful counterparts, expressed in Theorem~\ref{thm:solg}.
\end{Proposition}
\begin{proof}
  This follows by composition of D-algebraic series (see,
  e.g.,~\cite[Prop.~29]{BeBMRa17}). 
\end{proof}

Note that both series $\Pnn(t,y)$ (in the general and colourful cases)
are even D-finite \emm as functions
of $y$ and $\Rgf$,. The other two series $\Cnn(t,x,y)$ and $\Dnn(t,x,y)$
are clearly D-algebraic as functions of $x$, $y$ and $\Rgf$, and it is
natural to wonder if they might be D-finite. After all, in
Lemma~\ref{lem:Cat} we have met  a series that is written as the exponential of a
hypergeometric series and is not only D-finite, but even
algebraic. 

In the one-variable setting, it is known that if $F(t)$ is
D-finite, then $\exp(\int F(t))$ is D-finite if and only if $F(t)$ is
in fact algebraic~\cite{singer86}.
We can use this criterion to prove, for instance, that the series
$\Dnn(t,0,1)$ of Theorem~\ref{thm:solQ} is not D-finite as a function
of $\Rgf$. Indeed,   $\Dnn(t,0,1)=D(\Rgf)$ with
\begin{align*}
  D(r) &=\exp\left(\sum_{n\ge 0} \sum_{j=0}^{n}  \frac 1 {n+1}
{2n-j\choose n} {3n-j+1 \choose 2n-j} r^{n+1}\right)
\\
&= \exp\left(\sum_{n\ge 0} \frac{3n+2}{2(n+1)^2} {2n\choose n} {3n+1
  \choose 2n}r^{n+1}\right)
\\
&=\exp\left(\int F(r)\right)
\end{align*}
where
\[
F(r)=\sum_{n\ge 0} \frac{3n+2}{2(n+1)} {2n\choose n} {3n+1
  \choose 2n}r^{n}.
\]
Then $D(r)$ is D-finite if and only if $F(r)$ is algebraic. But this
is not the case, as the coefficient of $r^n$ in $F(r)$ is asymptotic
to $c\, 27^n/n$, which reveals a logarithmic singularity in $F(r)$
(see~\cite{flajolet-context-free}). The same argument proves that
$\Cnn(t,1,1)$ is not a D-finite function of $\Rgf$. 

In the colourful case (Theorem~\ref{thm:solg}), we have
\[
\Dnn(t,0,1)= \exp\left( \sum_{n\ge 0} \frac 1 {n+1} {2n+1\choose n}^2 \Rgf^{n+1}
\right),
\]
and a similar asymptotic argument proves that this cannot be a
D-finite function of $\Rgf$. The same holds for $\Cnn(t,1,1)$.

\subsection{Asymptotics}
As mentioned in the introduction, the series $\Rgf$ of
Theorem~\ref{thm:4} already occurred in the map literature, more
precisely in the enumeration of quartic maps equipped with a spanning
forest~\cite{mbm-courtiel}. Its singular structure has been studied in details,
and the first part of the following result is the case $u=-1$ of~\cite[Prop.~8.4]{mbm-courtiel}. 
As in~\cite[Def.~VI.1, p.~389]{flajolet-sedgewick}, we call
\emm $\Delta$-domain of radius $\rho$, any domain of the form
$$
\{z :|z|<r, z\not = \rho \hbox{ and } |\Arg (z-\rho)|> \phi\}
$$
for some $r>\rho$ and $\phi \in (0, \pi/2)$.

\begin{Proposition}\label{prop:asymptR}
   The series $\Rgf$ of Theorem~\ref{thm:4} has radius $\rho= \frac {\sqrt 3} {12\pi}$. It is analytic in a
     $\Delta$-domain of radius $\rho$, and the following estimate
     holds in this domain, as $t\rightarrow \rho$:
\[ 
      \Rgf(t) - \frac 1 {27} \sim 
\frac 1 6\, \frac{1-t/\rho}{\log (1-t/\rho)}.
 \]
Consequently, the $n$th coefficient of $\Rgf$ satisfies, as
$n\rightarrow \infty$,
$$
r_n:=[t^n]\Rgf \sim - \frac 1 6 \frac {\mu^n}{n^2 \log ^2 n}
$$
with $\mu=1/\rho= 4\sqrt 3 \pi$.
\end{Proposition}
Observe that this provides the asymptotic behaviour of the numbers
$q_n$ of Theorem~\ref{thm:4} since $q_n= -r_{n+2}/3$. The
correspondence between the singular behaviour of $\Rgf(t)$ near its
dominant singularity $\rho$ and the asymptotic behaviour of its
coefficients relies on Flajolet and Odlyzko's \emm singularity
analysis, of \gfs~\cite{flajolet-odlyzko,flajolet-sedgewick}. The
singular behaviour of $\Rgf$ near $\rho$ is obtained using the
inversion relation $\Omega(\Rgf(t))=t$, where the series $\Omega$,
given by~\eqref{omega4}, has radius $1/27$ and satisfies
\[
\Omega\left(\frac 1 {27}(1-\vareps)\right)= \frac{\sqrt 3}{12\pi}+\frac{\sqrt
  3}{54 \pi} \vareps\log \vareps + O(\vareps)
\]
as $\vareps \rightarrow 0$.

For general Eulerian orientations, we have a similar result.
\begin{Proposition}\label{prop:asympt-gen}
   The series $\Rgf$ of Theorem~\ref{thm:gen} has radius $\rho= \frac {1} {4\pi}$. It is analytic in a
     $\Delta$-domain of radius $\rho$, and the following estimate
     holds in this domain, as $t\rightarrow \rho$:
\[ 
      \Rgf(t) - \frac 1 {16} \sim 
\frac 1 4\, \frac{1-t/\rho}{\log (1-t/\rho)}.
 \]
Consequently, the $n$th coefficient of $\Rgf$ satisfies, as
$n\rightarrow \infty$,
$$
r_n:=[t^n]\Rgf \sim - \frac 1 4 \frac {\mu^n}{n^2 \log ^2 n}
$$
with $\mu=1/\rho= 4 \pi$.
\end{Proposition}
As above, this gives the asymptotic behaviour of the numbers
$g_n$ of Theorem~\ref{thm:4} since $g_n= -r_{n+2}/4$.

The proof {closely follows} the proof of Proposition~\ref{prop:asymptR}
given in~\cite[Sec.~8.3]{mbm-courtiel}, and we will not give any
details. The series $\Omega$ is now given
by~\eqref{omegaG}, has radius of convergence $1/16$ and satisfies
\[
\Omega\left(\frac 1 {16}(1-\vareps)\right)= \frac{1}{4\pi}+\frac{1}{16 \pi} \vareps\log \vareps + O(\vareps).
\] 
One key ingredient is that $t-\Rgf(t)$ has
non-negative coefficients, which simply follows from the fact that
this series equals $2t^2+4t^2\Ggf(t)$, by Theorem~\ref{thm:gen}.

\section{Final comments and perspectives}
\label{sec:final}

We have {exactly solved} the problem of counting planar Eulerian
orientations, both in the general and in the quartic case. Our proof,
based on a guess-and-check approach, should not stay the only
proof. One should seek a better combinatorial understanding of
our results. Can one explain why the series $\Rgf$ of Theorem~\ref{thm:4}
also appears in the enumeration of quartic maps~$M$ weighted by their
Tutte polynomial $\Tpol_M(0,1)$? Can one explain the forms of the
series $\Cgf$ and~$\Dgf$ in Theorems~\ref{thm:solQ} and~\ref{thm:solg}? What about
more general vertex degrees? Can one interpolate between the results
of Theorems~\ref{thm:4} and~\ref{thm:gen}, given that the
second also counts a subclass of quartic Eulerian orientations (those
with no alternating vertex)?  In this final section we discuss the {quest for} bijections, and some aspects of interpolation.

\subsection{Bijections}
Our results reveal an unexpected connection
between Eulerian orientations of quartic maps (counted by the specialization
$|\!\Tpol(0,-2)|$ of their Tutte polynomial if we do not force
the orientation of the root edge~\cite[Sec.~3.6]{welsh-book}) and the specialization $\Tpol(0,1)$
of slightly larger maps. Let us be more precise: one of the series
considered in~\cite{mbm-courtiel} is
\[
F(t)=\sum_{M \small{\hbox{ quartic}}}t^{\ff(M)} \Tpol_M(0,1),
\]
which, in the classical interpretation of the Tutte polynomial~\cite{Bollobas:Tutte-poly,tutte-dichromate}, counts
quartic maps $M$ equipped with an \emm internally inactive, spanning
tree. Observe that $t$ records here the number of faces, which exceeds
the number of vertices by 2. Then it is proved that
\[
F'(t)= 4\sum_{i\ge 1} \frac 1{i+1} {3i \choose i-1} {2i+1 \choose i}
\Rgf^{i+1},
\]
where $\Rgf$ is the series of Theorem~\ref{thm:4}. There is also an
interpretation of $F(t)$ in terms of spanning forests rather than spanning
trees, but then some forests have a negative contribution:
\[
F(t)=\sum_{{M \small{\hbox{ quartic}}} \atop{F \small\hbox{ forest}}}
t^{\ff(M)} (-1)^{\cc(F)-1}
\]
where $\cc(F)$ is the number of connected components of the forest
$F$. One of the advantages of this description in terms of forests is
that it gives a direct interpretation of $t-\Rgf$. Indeed, if we
restrict the summation to forests not containing the root edge, then
we obtain a new series, denoted $H(t)$ in~\cite{mbm-courtiel}, which
satisfies
\[
H'(t)=2(t-\Rgf).
\]
Comparing with Theorem~\ref{thm:4} leads to the following statement: the number of Eulerian orientations of quartic maps
with $n$ faces  is $(n+1)/6$ times the {(signed) number} of
quartic maps with $n+1$ faces equipped with a spanning forest not
containing the root edge,  every forest $F$ being weighted by
$(-1)^{\cc(F)-1}$. This is illustrated in Figure~\ref{fig:bij} for
$n=3$. 

\begin{figure}[ht]
  \centering
   \scalebox{0.9}{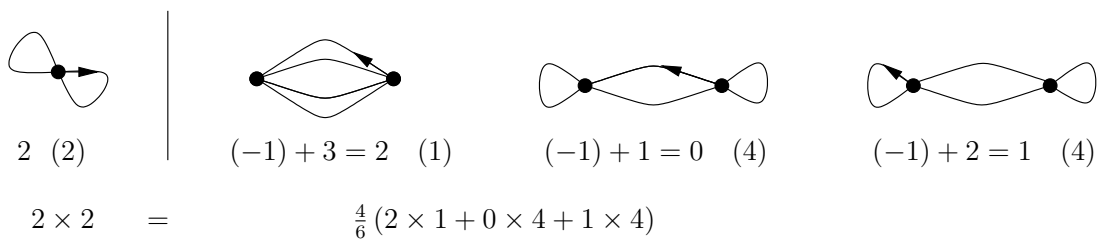}
  \caption{Left: the only rooted quartic graph with $1$ vertex has 2
    Eulerian orientations (the orientation of the root is forced) and
    2 embeddings as a rooted map (with 3 faces). Right: the three
    rooted quartic graphs with 2 vertices, shown with the (signed)
    number of spanning forests avoiding the root edge. The number of
    embeddings as rooted planar maps is shown between parentheses. }
  \label{fig:bij}
\end{figure}

There is also an interpretation (and generalization) of $t-\Rgf(t)$ in terms of certain
trees~\cite[Sec.~5.1]{mbm-courtiel}. It involves a
parameter $u$, which is $-1$ for our series $\Rgf$.  In the forest
setting, $u$ counts the number of connected components (minus 1).

\begin{Proposition}
  Consider rooted plane ternary trees with leaves of two colours (say
  black and white, see Figure~\ref{fig:R4}). Define the \emm charge, of such a tree {to be} the
  {number of white leaves minus the number of black leaves. Call
  a tree of charge $1$ \emm balanced,}. Then the series $t-\Rgf(t)$ of
  Theorem~\ref{thm:4} counts, by the number of white leaves, balanced trees in which no proper subtree
  is balanced.

More generally, let $R(t,u)\equiv R$ be the only power series in $t$ with
constant term $0$ satisfying
\[
R= t+ u \sum_{n\ge 1} \frac 1 {n+1}{2n\choose n}{3n\choose n} R^{n+1},
\]
so that $R(t,-1)= \Rgf(t)$. Then $(R-t)/u$ counts balanced trees by the number of
white leaves, with an additional
weight $(u+1)$ per proper balanced subtree.
\end{Proposition}
Here, a \emm subtree, of a tree $T$ consists of  a vertex of $T$ and all its
  descendants, and  is \emm proper, if the chosen vertex is neither the
  root of $T$ nor a leaf.
{This} proposition is illustrated in Figure~\ref{fig:R4}.

\begin{figure}[ht]
\setlength{\captionindent}{0pt}
     \scalebox{0.9}{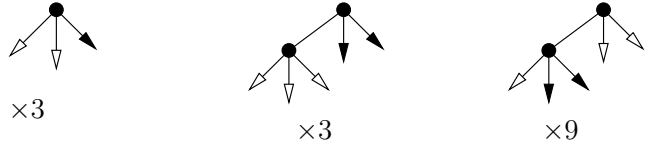}
   \caption{The trees with 2 and 3 white leaves involved in the
     expansion of $t-\Rgf=3t^2+12t^3+O(t^4)$, for the series $\Rgf$ of
     Theorem~\ref{thm:4}. The
      multiplicities indicate the number of     embeddings in the plane.}
   \label{fig:R4}
\end{figure}

In the case of general Eulerian orientations (Theorem~\ref{thm:gen}), we also have a similar
combinatorial interpretation and generalization of  $t-\Rgf(t)$.

\begin{Proposition}
  Consider rooted plane binary trees  with edges of two colours (say
  solid and dashed, see Figure~\ref{fig:RG}). Define the \emm charge, of such a tree {to be the number of solid edges minus the number of dashed edges}. Call a tree of charge $0$ {\emm balanced,}. Then the series $t-\Rgf(t)$ of
  Theorem~\ref{thm:gen} counts, by leaves, balanced trees in which no proper subtree  is balanced.

More generally, let $R(t,u)\equiv R$ be the only power series in $t$ with
constant term $0$ satisfying
\beq\label{Ru}
R= t+ u \sum_{n\ge 1} \frac 1 {n+1}{2n\choose n}^2 R^{n+1},
\eeq
so that $R(t,-1)= \Rgf(t)$. Then $(R-t)/u$ counts balanced trees by
the number of leaves, with an additional
weight $(u+1)$ per proper balanced subtree.
\end{Proposition}
This proposition is illustrated in Figure~\ref{fig:RG}.

\begin{figure}[ht]
\setlength{\captionindent}{0pt}
     \scalebox{0.9}{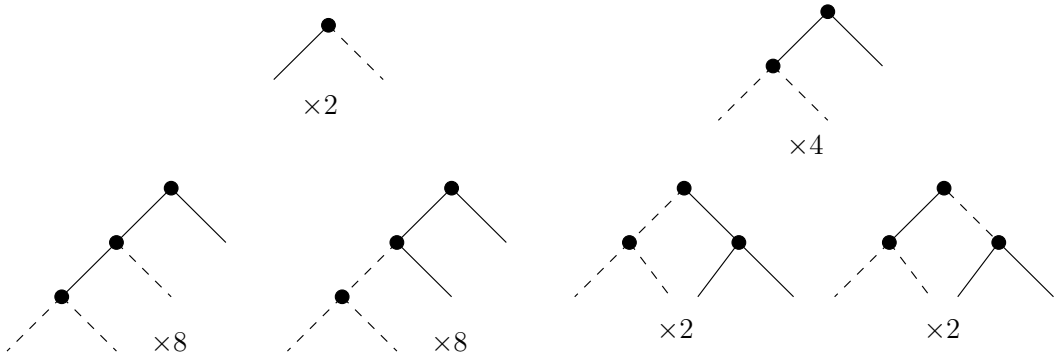}
   \caption{The trees with at most 4  leaves involved in the
     expansion of $t-\Rgf=2t^2+4t^3+20t^4+O(t^5)$, for the series $\Rgf$ of
     Theorem~\ref{thm:gen}. The  multiplicities takes into account the
     number of embeddings in the plane and the exchange of the two
     colours.}
   \label{fig:RG}
\end{figure}
\begin{proof}
Define a \emm marked, balanced tree as a balanced tree in which a
number of inner vertices are marked, in such a way that:
\begin{itemize}
\item the root vertex is marked (unless the tree consists of a single
  vertex)
 \item the subtree attached at any marked vertex is balanced.
\end{itemize}
  Let $\bR(t,u)$ be the \gf\ of marked balanced trees  with a weight $t$ per
  leaf and a weight $u$ per marked vertex.  We claim that $\bR$
  satisfies~\eqref{Ru}. Indeed, take a marked balanced tree
  with at least one edge, and consider the tree obtained by deleting all
  subtrees attached to a (non-root) marked vertex. Then this tree must
  be balanced. If it has $n$ inner vertices, it can be chosen and
  coloured in
\[ \frac 1 {n+1}{2n \choose n} {2n \choose n} 
\]
ways: the Catalan number accounts for the choice of the tree, and the
second binomial coefficient for the colouring of its $2n$ edges. To
reconstruct the marked balanced tree, we now need to attach to each of
the $n+1$ leaves a marked balanced tree, and this
gives~\eqref{Ru}. Hence the series $\bR(t,u)$ coincides with $R(t,u)$.

Now consider a balanced tree. The total weight of all marked trees
that can be constructed from it by marking certain vertices is
$u(u+1)^{b}$, where $b$ is the number of proper balanced
subtrees. This completes the proof.
\end{proof}
The problem of understanding {these equidistributions bijectively} is
wide open. Let us mention that  deep connections are known to
  exist between certain families of
  orientations (e.g., acyclic) of a graph and certain families of subgraphs (e.g., spanning
  forests) of the same graph (see~\cite{bernardi-tutte} for a survey, and references therein).

\medskip
Let us finish with another bijective question. There exist two
  main bijections that transform Eulerian maps into trees: one of them
  takes the dual bipartite map,
  and transforms it into a mobile-tree using the distance labelling
  of the vertices~\cite{bouttier-mobiles}. This is the Bouttier--Di-Francesco--Guitter
  bijection that we have generalised in Section~\ref{sec:bij} to more
  general labellings, and thus to Eulerian \emm orientations, (rather
  than Eulerian maps). The second classical bijection, due to
  Schaeffer~\cite{Sch97}, transforms Eulerian maps into \emm blossoming
  trees,. Underlying this construction is a canonical Eulerian
  orientation of the map. Is there an extension of this bijection to
  all Eulerian orientations?

\subsection{Interpolating between quartic Eulerian orientations and
  general Eulerian orientations}
\label{sec:int}

Given that the form of our solution for general Eulerian orientations
is so similar to that for quartic Eulerian orientations, one may
wonder whether these are two special cases of a more general
series. In a forthcoming paper we describe two possible ways to
simultaneously generalise  $\Ggf(t)$ and $\Qgf(t)$.  The first series 
{that} we
consider counts general Eulerian orientations by edges and vertices.
This is an obvious generalisation of $\Ggf(t)$, which
only records the number of edges. Moreover,  $\Qgf(t)$
can be extracted from this refined \gf\  by directly utilising the fact
that quartic Eulerian orientations form a subclass of general Eulerian
orientations (those having, in a sense, many vertices). The second generalisation  concerns
labelled quadrangulations, and interpolates between the series
$\Qgf(t)$ and $\Qc(t)=2\Ggf(t)$ by keeping track of the number of quadrangles
that {only contain two labels} (such quadrangles are forbidden in
colourful quadrangulations). This corresponds to the six vertex model
discussed in Section~\ref{sec:ice}.

\bigskip
\noindent {\bf Acknowledgements.} We are grateful to Tony Guttmann for
putting the authors in contact with each other and organising AEP's
visit to the University of Bordeaux in June 2017.  We also thank him
for useful comments on the manuscript.  The authors acknowledge many
interesting discussions with Nicolas Bonichon about Eulerian orientations. {We would like to thank Paul Zinn-Justin for helping us understand Kostov's solution to the six vertex model on a random lattice. We are grateful to J\'er\'emie Bouttier for pointing us to several articles relating to the bijection in Section~\ref{sec:bij}. Finally, we thank the anonymous referees for their helpful comments that have improved the paper.}


\bibliographystyle{plain}
\bibliography{oe}
\end{document}